\documentclass{siamltex}

\usepackage{graphicx}

\usepackage{amssymb}
\usepackage{amsmath}
\usepackage{bm}

\def\og{\leavevmode\raise.3ex\hbox{$\scriptscriptstyle\langle\!\langle$~}}
\def\fg{\leavevmode\raise.3ex\hbox{~$\!\scriptscriptstyle\,\rangle\!\rangle$}}

\newcommand{\beqa}{\begin{eqnarray}}
\newcommand{\eeqa}[1]{\label{#1}\end{eqnarray}}
\newcommand{\beq}{\begin{equation}}
\newcommand{\eeq}[1]{\label{#1}\end{equation}}

\newcommand\ep{\varepsilon}

\newcommand\norm[2]{\left\Vert #1 \right\Vert_{#2}}
\newcommand\twoscale{\overset{_2}{\rightharpoonup}}

\newcommand{\spec}[1]{\sigma{\left( #1 \right)}}

\newcommand\dblint[2]{\int_{#1}{\int_{Q}{#2}  \ \mathrm{d}y} \mathrm{d}x}

\def\phi{\varphi}

\def\demifleche{\rightharpoonup}

\def\*fleche{\buildrel *\over\demifleche}

\def\tol2{\buildrel\hbox{$L^2$}\over\longrightarrow}

\def\toto{\leaders\hbox to 5mm{\hfil.\hfil}\hfill}

\title{Spectral analysis of one-dimensional high-contrast elliptic problems 
with periodic coefficients} 

\author{K. D. Cherednichenko\thanks{School of Mathematics,Cardiff University, Senghennydd Road, Cardiff, CF24 4AG, 
United Kingdom ({\tt cherednichenkokd@cardiff.ac.uk})} \and S. 
Cooper\thanks{School of Mathematics,Cardiff University, Senghennydd Road, Cardiff, CF24 4AG, 
United Kingdom} \and S. Guenneau\thanks{Institut Fresnel, Aix-Marseille Universit\'{e} (UMR CNRS 7249), Domaine Universitaire de Saint-J\'er\^ome F13397, Marseille cedex 20, France}}

\medskip

\begin{document}

\maketitle

\begin{abstract}
We study the behaviour of the spectrum of a family of one-dimensional operators with 
periodic high-contrast coefficients as the period goes to zero, which may represent 
{\it e.g.}  the elastic or electromagnetic response of a two-component composite medium. Compared to the standard 
operators with moderate contrast, they exhibit a number of new effects due to the 
underlying non-uniform ellipticity of the family. The effective behaviour of such media in the 
vanishing period limit also differs notably from that of multi-dimensional models 
investigated thus far by other authors, due to the fact that neither component of 
the composite forms a connected set. We then discuss a modified problem, where the equation 
coefficient is set to a positive constant on an interval that is independent of the period. 
Formal asymptotic analysis and numerical tests with finite elements suggest the existence of localised eigenfunctions (``defect modes''), whose eigenvalues situated in the gaps of the limit spectrum for the 
unperturbed problem.

\end{abstract}





\pagestyle{myheadings}
\thispagestyle{plain}
\markboth{K. D. CHEREDNICHENKO, S. COOPER and S. GUENNEAU}{SPECTRAL ANALYSIS OF HIGH-CONTRAST PROBLEMS}

\section{Introduction}

\subsection{The general context for the problem in hand}

The description of the effective behaviour of high-contrast composites ("high-contrast homogenisation") has been of particular interest in the analysis and applied communities over the last decade.  The analytical part of the related literature starts with the work \cite{Zh2000}, which developed in detail some earlier ideas of 
\cite{Allaire} concerning the use of ``two-scale convergence'' for the analysis of the limit behaviour of the 
boundary-value problem
\[
-{\rm div}({\mathcal A}^\varepsilon(x/\varepsilon)\nabla u)=f,\ \ \ \ f\in L^2(\Omega),\ \ \ u\in H_0^1(\Omega),\ \ \ {\mathcal A}^\varepsilon=\varepsilon^2\chi_0I+\chi_1I,\ \ \varepsilon>0,
\]
where $\Omega\subset{\mathbb R}^n$ is a bounded domain, and $\chi_0,$ $\chi_1$ are the indicator 
functions of $[0,1)^n$-periodic sets in ${\mathbb R}^n$ such that $\chi_0+\chi_1=1.$
 
 Several contributions to the high-contrast homogenisation followed: in the linear and non-linear, scalar and vector 
contexts, with various sets of assumptions about the underlying geometry of the composite.  With applications 
mainly in solid mechanics and electromagnetism, high-contrast media have served as a theoretical ground for a number of effects observed in physics experiments, in particular those related to photonic band-gap materials and cloaking metamaterials (\cite{RamakrishnaGrzegorczyk}). The range of techniques developed in these contexts and their applications continue their rapid expansion, and the present paper is one contribution aimed at addressing some aspects that have thus far been left out of the scope of the related research. 

More specifically, we approach the question of the analysis of the spectral behaviour of high-contrast composites in the case when the component represented by the  function $\chi_1$ (the "matrix" of the composite) is disconnected in ${\mathbb R}^n.$ Clearly, this is always the case in one dimension ($n=1$), which is the situation we study in the present article.

\subsection{Problem setup}\label{setupsubsection}

We consider solutions $u$ to the following family of elliptic problems on an interval 
$(a,b)\subset{\mathbb R}:$  
\beq
{A}^\varepsilon u-\lambda u=f,\ \ \ f\in L^2(a,b),\ \ \ \varepsilon>0,\ \ \lambda\in{\mathbb C},
\eeq{ourproblem}
where the operators ${A}^\varepsilon$ are given by the closed bilinear form
\beq
({A}^\varepsilon u,v)=\int_a^bp(x/\varepsilon)\bigl(\varepsilon^2\chi_0(x/\varepsilon)+\chi_1(x/\varepsilon)\bigr)u'(x)\overline{v'(x)}dx,\ \ \ u,v\in {\mathfrak H}.
\eeq{bilinearform}
Here $p=p(y)>0$ is a $1$-periodic function in ${\mathbb R}$ such that $p,p^{-1}\in L^\infty(0,1),$ the functions $\chi_0$ and $\chi_1$ are the indicator functions of 1-periodic open sets $F_0$ and $F_1$ such that 
$\overline{F}_0\cup\overline{F}_1={\mathbb R},$ and ${\mathfrak H}$ denotes a closed linear subspace of $H^1(a,b)$ that contains $C_0^\infty(a,b).$
We make no assumptions regarding boundedness of the interval $(a,b),$ in particular it may coincide with 
the whole space ${\mathbb R}.$

In applied contexts the problem (\ref{ourproblem}) corresponds to, {\it e.g.}, the study of wave propagation in a layered 2D or 3D composite 
structure where $f=0,$ $\lambda>0.$ In what follows we study the spectrum $S^\varepsilon$ of the problem (\ref{ourproblem}), {\it i.e.} the set of 
values of $\lambda$ for which ${A}^\varepsilon-\lambda I$ does not have a bounded inverse in $L^2(a,b).$ Throughout the article we employ the notation $\sigma(A)$ for the spectrum of an operator $A,$ and the notation $Q$ for the ``unit cell'' $[0,1)$ 
whenever we describe the behaviour with respect to the ``physical'' variables $x,y.$ We continue writing $[0,1)$ for the 
``Floquet-Bloch dual'' cell when we refer to the domain of the quasimomentum $\theta.$

\subsection{Our strategy for the analysis of (\ref{ourproblem})}

It has been well understood in the existing literature on the subject (see \cite{AllaireConca}, \cite{Zh2000}, \cite{Zh2005}), that in the analysis of 
convergence of spectra of families of differential operators with periodic rapidly oscillating coefficient, one has to deal with two 
distinct issues: the lower 
semicontinuity of the spectra in the sense of Hausdorff convergence of sets, and the possibility of spectral pollution, the lack of which is often
referred to as ``spectral completeness''. The former issue, which in the wider spectral analytic context has been looked at from a more general perspective 
(see {\it e.g.} \cite{ChanWildeLindner}), is usually dealt with by proving first a variant of strong resolvent convergence. In the case of periodic operators involving multiple scales, one typically makes use of the so-called ``two-scale convergence'' 
(see {\it e.g.} \cite{Nguetseng}, \cite{Allaire}, \cite{Zh2000}). In the present paper we follow this general approach in proving the related 
lower semicontinuity statements both for the whole-space problem and for the problem in a bounded interval. It should be pointed out that 
this first part of the analysis of spectral convergence is not completely independent from the subsequent study of spectral completeness: 
unless some assumptions are made concerning the geometry of the periodic composite in question (see {\it e.g.} \cite{Zh2000}), one may not get 
the best possible ``lower bound'' for the limit spectrum. It has been noticed that, in order to capture the behaviour with respect to all Bloch components 
in the limit as $\varepsilon\to0,$ it is preferable to use an advanced, ``multi-cell'' version, of the standard two-scale convergence; 
see {\it e.g.} \cite{AllaireConca}, \cite[Chapter 5]{Cooper}, where this more refined approach is adopted. It is a version of this last, more 
detailed, procedure that we adopt in the present article.

    In the proof of spectral completeness, a natural strategy seems to try and analyse the relative strength of different Bloch components 
in a given (convergent) sequence of eigenfunctions. This idea has been elaborated in \cite{AllaireConca} in the specific context of 
``high-frequency'' homogenisation with the use of what the authors refer to as the ``Bloch measures''. A combination of a compactness argument in 
the related space of measures and a special ``slow-variable modulation'' construction then yields the simultaneous convergence of the given sequence to 
a limit eigenfunction and of the associated eigenvalues. In the present work we suggest an alternative approach (see Section \ref{boundedintervalsection}) 
to the convergence of 
eigenfunctions, which we believe is closer in spirit to the idea of ``spectral compactness'', {\it i.e.} compactness of eigenfunctions 
in a norm-preserving topology. Our approach is based on the idea that once one has control of the behaviour of eigenfunctions in 
the orthogonal complement to the space spanned by the limit eigenfunctions, one can immediately pass to the limit, as $\varepsilon\to0,$ in the 
weak formulation of the original family of eigenvalue problems. This idea allows us to cover the analysis of spectral convergence for 
a wide range of operator families, including those considered by \cite{AllaireConca},\cite{Zh2000}, \cite{Sm2009}, \cite[Chapter 4]{Cooper}.  

The key element in our analysis, which allows us to implement the above idea is Proposition \ref{prop0} below (see Section \ref{rigorousstatementsection}), or equivalently Proposition \ref{lem1}. These statements establish 
a uniform version of the Poincar\'{e}-type inequality between the projection of a given function onto the ``poorly behaving'' subspace and the $L^2$-norm of its derivative on the part of the domain where solutions of the eigenvalue problem can be shown to be {\it a priori} small as $\varepsilon\to0.$ Different versions of the same idea have appeared in a number of other contexts, serving a similar purpose of ``compensating'' somehow the apparent loss of compactness in the problem, for example, in the form of Korn inequality in elasticity (see {\it e.g.} \cite{DuvautLions}, and 
also \cite{Zh_sing} for its multiscale versions), in the form of the so-called ``energy method'' in classical homogenisation (see \cite{MuratTartar}),  and, more recently, in the form of a ``generalised Weyl decomposition'' for problems with degeneracies (see \cite{SmKam}). For nonlinear variants of the same idea, the reader may be referred to the ``geometric rigidity'' (see \cite{FJM}) and ``${\mathcal A}$-quasiconvexity'' (see \cite{Mulleretal}).

For an easier introduction to the problem, in what follows we start with the analysis of the problem (\ref{ourproblem}) in the whole-space 
case, $(a,b)={\mathbb R},$ see Section \ref{wholespacesection}. While a version of the the compactness argument developed in the 
bounded-interval setting (see Section \ref{boundedintervalsection}) applies here as well (once complemented by a suitable Weyl-sequence 
argument), we present a different argument, based on some ideas of \cite[Chapter 5]{Cooper}, where the spectral analysis is carried out 
in a more challenging setting of the Maxwell system.

Throughout the article we assume for simplicity that the restriction of $\chi_0$ to the periodicity cell $[0,1)$ is the
indicator function of an open interval $(\alpha,\beta),$ which we also denote by $Q_0.$ We use the notation $Q_1$ for 
the interior of the complement of $Q_0$ to the interval $(0,1).$

\section{Limit analysis for the whole space}\label{wholespacesection}

In this section we consider the case $(a,b)={\mathbb R}.$ One well-known procedure for calculating $S^\varepsilon$ is the Floquet-Bloch decomposition (\cite{BLP}) following the rescaling $y=x/\varepsilon.$ Then, for $\theta\in(0,1]$ the sequence of eigenvalues $\lambda=\lambda(\theta)$ corresponding to 
$\theta$-quasiperiodic solutions to the Floquet-Bloch problem on the interval $(0,1)$  associated to the differential expression 
$(p(\varepsilon^2\chi_0+\chi_1)u')'$ is obtained by solving the 
dispersion equation 
\[
\frac{1}{2}\biggl(\frac{1}{\varepsilon}+\varepsilon\biggr)\sin\Bigl(\varepsilon\sqrt{\lambda}(\alpha-\beta+1)\Bigr)
\sin\Bigl(\sqrt{\lambda}(\alpha-\beta)\Bigr)\ \ \ \ \ \ \ \ \ \ \ \ \ \ \ \ \ \ \ 
\]
\[
\ \ \ \ \ \ \ \ \ \ \ \ \ \ \ \ \ \ \  \ \ \ \ \ +\cos\Bigl(\varepsilon\sqrt{\lambda}(\alpha-\beta+1)\Bigr)\cos\Bigl(\sqrt{\lambda}(\alpha-\beta)\Bigr)=\cos(2\pi\theta).
\]
Passing to the limit in the above equation as $\varepsilon\to0$ yields
\beq
\frac{1}{2}(\alpha-\beta+1)\sqrt{\lambda}\sin\Bigl(\sqrt{\lambda}(\alpha-\beta)\Bigr)
+\cos\Bigl(\sqrt{\lambda}(\alpha-\beta)\Bigr)=\cos(2\pi\theta).
\eeq{limiteq}
By varying $\theta$ as indicated we obtain (for $\alpha=1/4,$ $\beta=3/4$) the set shown in Fig. \ref{limitspectrum}. 

\begin{center}
\begin{figure}[h]
\centering
\includegraphics[scale=0.3]{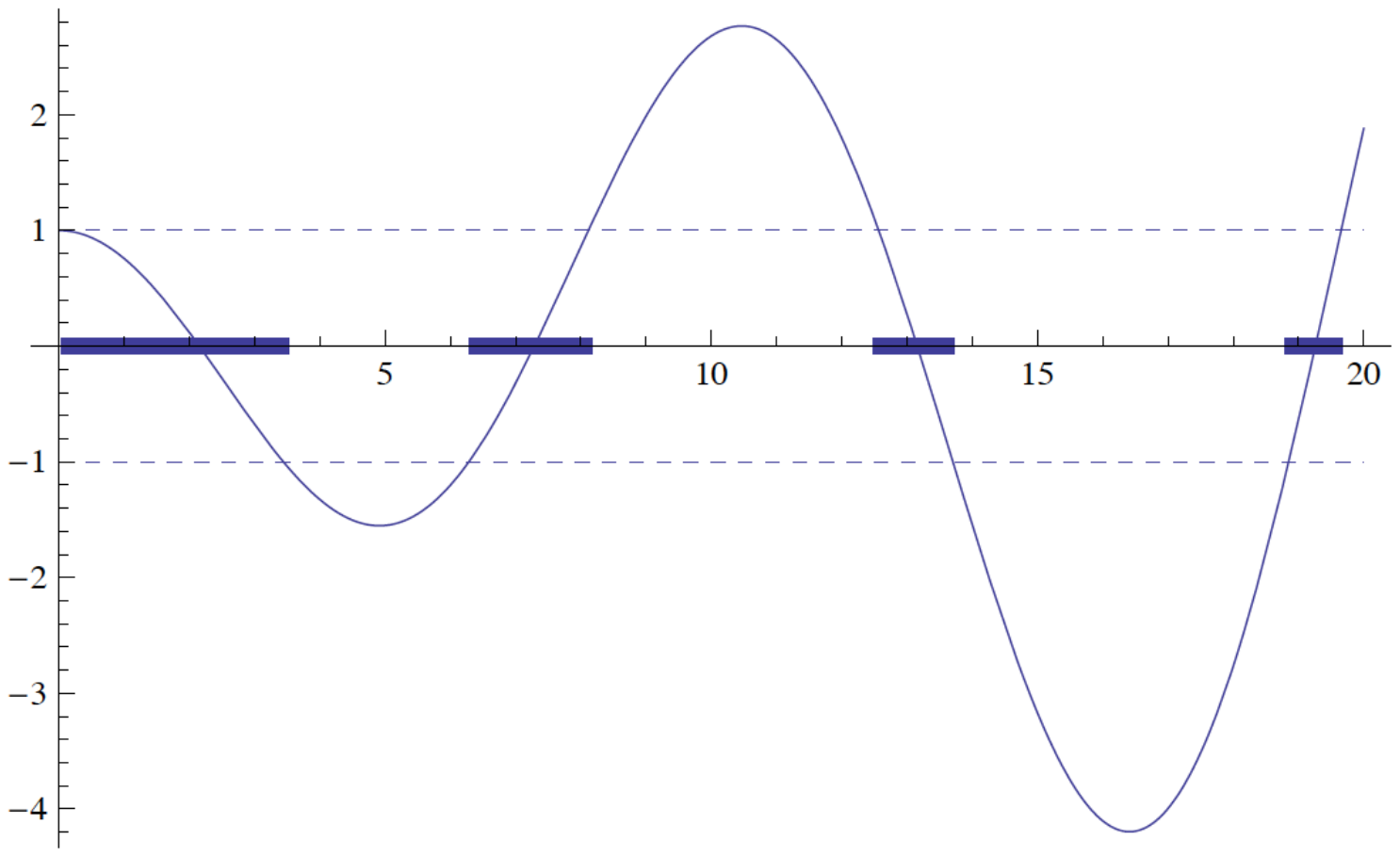}
\caption{The square root of the limit Bloch spectrum. The oscillating solid line is the graph of the function
$f(t)=\cos(t/2)-t\sin(t/2)/4,$ where $t$ represents $\sqrt{\lambda}$ in the formula (\ref{limiteq}) with $\alpha=1/4,$ 
$\beta=3/4.$ The square root of the spectrum is the union of the intervals indicated by bold lines. }
\label{limitspectrum} 
\end{figure}
\end{center}   

Our first result 
is the following theorem.
\begin{theorem}\label{maintheoremwholespace}
Let $(a,b)={\mathbb R}.$ Then the set $\lim_{\varepsilon\to0} S^\varepsilon$ is given by the union 
of solution sets for the equation (\ref{limiteq}) for all $\theta\in[0,1).$  
\end{theorem}

For $\theta \in [0,1)$, we denote by $H^1_{\theta}(Q)$ the space of functions $u \in H^1(Q)$  that are 
$\theta$-quasiperiodic, {\it i.e.} such that 
$v(y)=\exp(2\pi{\rm i}\theta y)u(y),$ $y\in Q,$ for some\footnote{As the notation $H^1_0(Q)$ is usually reserved for the space of $H^1(Q)$ functions vanishing on the boundary of $Q,$ we denote by $H^1_{\#}(Q)$ the space $H^1_{\theta}(Q)$ when $\theta=0.$} 
$u\in H^1_{\#}(Q)$. We also denote
\beq
V(\theta):=\left\{v\in H^1_{\theta}(Q):p (y)v'(y)=0{\ {\rm for}\ }y \in Q_1\right\}.
\eeq{Vthetadef}


Consider an operator $A(\theta)$ such that
\[
(A(\theta)u,\varphi)=\int_{Q_0}p(y)u'(y)\overline{\varphi'(y)}\mathrm{d}y\ \ \ \forall \,\varphi\in V(\theta),\ \ u\in{\rm dom}(A(\theta))\subset V(\theta), 
\]
defined on the maximal possible domain ${\rm dom}(A(\theta)).$  Henceforth $(\cdot,\cdot)$ denotes the usual inner product in $L^2(Q).$ By a standard argument 
(see {\it e.g.} \cite{Kato}) such an operator exists, is unique and, under the adopted conditions on the coefficient $p,$ is self-adjoint and has compact inverse (except for the case $\theta=0,$ when it has compact inverse as an operator on 
$V(\theta)\ominus {\mathbb C}$).   
Therefore the spectrum $\spec{A(\theta)}$ is discrete and unbounded, {\it i.e.} it consists of eigenvalues 
$0\le\lambda_1(\theta)\le\lambda_2(\theta)\le \ldots,$
of finite multiplicity with eigenfunctions  $v^k(\theta)=v^k(\theta,y).$ The eigenfunctions corresponding to different eigenvalues are automatically orthogonal in 
$L^2(Q).$ We also carry out the orthogonalisation process on those eigenfunctions that correspond to the same eigenvalue, and normalise each eigenfunction so that  
$\Vert v^k(\theta)\Vert_{L^2(Q)}=1,$ for all $\theta\in[0,1),$ $k\in{\mathbb N}.$ 

Our aim is to show that the limit set $\lim_{\varepsilon\to0}S^\varepsilon$ coincides with the union of the spectra of the operators $A(\theta),$ 
$\theta\in[0,1),$ which, in turn, are described by the ''dispersion relation'' (\ref{limiteq}). 

In order to demonstrate first that the latter is included in the former, for each $N\in{\mathbb N},$ we define an ``intermediate'' operator 
$A_N$ in $L^2(NQ),$ whose spectrum is contained in $\lim_{\varepsilon\to0}S^\varepsilon$ and contains the spectrum of each of the operators 
$A(\theta),$ $\theta=j/N,$ $0\le j\le N-1.$ The details of this argument, 
which relies on a procedure that we refer to as ``$NQ$-periodic homogenisation'', are given in Appendix A.

An essential component of the proof of the converse inclusion is the following lemma.

\begin{lemma}
\label{contV}
For any given $\theta \in [0,1)$ and $\phi \in V\left(\theta\right)$, let $\theta_\ep \in [0,1)$ be such that $\theta_\ep \rightarrow \theta$ as $\ep \rightarrow 0$. Then there exist $\phi_\ep \in V\left(\theta_\ep\right)$ such that $\phi_\ep \rightarrow \phi$ strongly in $H^1(Q)$ as $\ep \rightarrow 0$.
\end{lemma}

\begin{proof}
For $\theta \in [0,1)$ 
the space $V(\theta)$ consists of functions that are 
$\theta$-quasiperiodic and constant in each connected component of $Q_1,$ that is for any $\phi\in V(\theta)$ one has $\phi(y) = \eta({\theta},y)c+v(y),$ where 
$c \in \mathbb{C}$, $v \in H^1_0(a,b),$  and
\begin{equation*}
\eta(\theta,y):=\left\{\begin{array}{lll}1, & y\in[0,a), \\ \left(\exp(2\pi{\rm i}\theta)-1\right)(b - a)^{-1}(y-a)+1, & y \in[a,b], \\
\exp(2\pi{\rm i}\theta), & y\in(b,1).\end{array}\right.
\end{equation*}
For each value of $\varepsilon$ we now define $\phi_\ep$ by the formula $\phi_\ep(y)= \eta(\theta_\ep,y)c+v(y),$ 
$y\in Q.$ Notice that by construction $\phi_\ep \in V\left( \theta_\ep \right)$ and, since $\eta$ is uniformly continuous with 
respect to $\theta$, one has $\phi_\ep \rightarrow \phi$ strongly in $H^1(Q).$
\end{proof}

We next show that for a sequence $\lambda_\ep\in\spec{A_\ep}$ such that $\lambda_\ep\rightarrow\lambda$ the inclusion $\lambda\in\spec{A\left(\theta\right)}$ holds for some $\theta\in[0,1)$. 

\begin{theorem}\label{spcompthm}
\noindent Let $\lambda_\ep \in \spec{A_\ep}$ such that $\lambda_\ep\rightarrow \lambda$. Then there exist 
$\theta\in[0,1)$ and $u \in H^1_{\theta}(Q),$ $u\neq0,$ such that
\begin{equation}
\label{spcomres}
\int_a^b p(y)u'(y)\overline{\phi'(y)}\mathrm{d}y=\lambda\int_0^1u(y)\overline{\phi(y)}\mathrm{d}y,
\quad\quad\forall\phi\in 
V\left(\theta \right).
\end{equation}

\end{theorem}

\begin{proof}
Since $\lambda_\ep \in \spec{A_\ep}$, by the Floquet-Bloch decomposition, there exists $u_\ep \in H^1_{\theta_\ep}(\ep Q),$ 
$u\neq 0,$ such that

\begin{equation}\label{wprobbloch}
\int_{\ep Q}p(x/\varepsilon)\bigl(\varepsilon^2\chi_0(x/\varepsilon)+\chi_1(x/\varepsilon))\bigr)u_\ep'(x)\overline{\phi'(x)}\mathrm{d}x=\lambda_\ep \int_{\ep Q}u_\ep(x)
\overline{\phi(x)}\mathrm{d}x
\end{equation}
for all $\phi \in H^1_{\theta_\ep}(\ep Q).$ Rescaling the formulation \eqref{wprobbloch} with $y=x/\varepsilon$  yields the existence of $u_\ep \in H^1_{\theta_\ep}(Q),$ 
$\norm{u_\ep}{L^2(Q)}=1,$ such that

\begin{equation}\label{wprobbloch1}
\ep^{-2}\int_{Q_1} p(y)u_\ep'(y)\overline{\phi'(y)}\mathrm{d}y+\int_{Q_0}p(y)u_\ep(y)\overline{\phi(y)}\mathrm{d}y=
\lambda_\ep\int_Qu_\ep(y)\overline{\phi(y)}\mathrm{d}y
\end{equation}
for all $\phi\in H^1_{\theta_{\ep}}(Q).$

The sequence $\theta_\ep$ is bounded and therefore there exists some $\theta \in [0,1]$ such that, up to a subsequence which we do not relabel, $\theta_\ep \rightarrow \theta$. Without loss of generality, if $\theta=1$ we set $\theta=0,$ so that $\theta\in[0,1).$ By substituting $\phi=u_\ep$ in \eqref{wprobbloch1}, the sequence $u_\ep$ satisfies the bounds
\beq
\bigl\Vert\chi_1u_\ep'\bigr\Vert_{L^2(Q)}\le C\varepsilon,\ \ \ \ \ \ \ \ \ \ \bigl\Vert\chi_0 u_\ep'\bigr\Vert_{L^2(Q)}\le C. 
\eeq{waprioribounds}
with a constant $C>0$ independent of $\ep.$

Due to the weak compactness of bounded sets in $H^1(Q),$ the bounds \eqref{waprioribounds}, along with $\norm{u_\ep}{L^2(Q)}=1,$ imply that, up to extracting a subsequence, $u_\ep$ converge weakly in $H^1(Q)$, and therefore strongly in $L^2(Q),$ to some $u_0 \in H^1(Q),$ $\norm{u_0}{L^2(Q)}=1$. Clearly, for 
$w_\ep(y):=\exp(-2\pi{\rm i}\theta_\ep y)u_\ep(y)$ one has $w_\ep\in H^1_{\#}(Q),$ and the uniform convergence of  
$\exp(2\pi{\rm i}\theta_\ep y)$ to $\exp(2\pi{\rm i}\theta y)$ as $\varepsilon\to0$ implies that $w_\ep$ converge weakly in $H^1(Q)$ to $w_0$ given by the formula 
$w_0(y)=\exp(-2\pi{\rm i}\theta y)u_0(y),$ so that $u_0\in H^1_{\theta}(Q).$ 
Furthermore, \eqref{waprioribounds} implies that $\chi_1 u_\ep'\rightarrow 0$ strongly in $L^2(Q),$ hence $u_0\in V\left(\theta\right).$ 

In order to show that $u_0$ satisfies the limit identity \eqref{spcomres}, for a fixed $\phi_0 \in V\left( \theta\right),$ let $\phi_\ep \in V\left( \theta_\ep \right)$ be given by Lemma \ref{contV}. Substituting $\phi_\ep$ in \eqref{wprobbloch1}, we obtain
\beq
\int_a^b p(y)u_\ep'(y)\overline{\phi_\ep'(y)}\mathrm{d}y=\lambda_\ep\int_0^1u_\ep(y)\overline{\phi_\ep(y)}\mathrm{d}y.
\eeq{wprobbloch2}
By virtue of the facts that $\phi_\ep \rightarrow \phi_0$ strongly in $H^1(Q)$ and $u_\ep \rightharpoonup u_0$ weakly 
in $H^1(Q)$, passing to the limit $\ep \rightarrow 0$ in \eqref{wprobbloch2} immediately implies \eqref{spcomres}.
\end{proof}

The above ``limit spectrum'' $\lim_{\varepsilon\to0}S^\varepsilon$ is strictly larger than 
the set obtained by the two-scale analysis of the operator $A^\varepsilon$ of the paper 
\cite{Zh2000}. In particular, the spectrum of the homogenised operator obtained in \cite{Zh2000} coincides with 
$\{\lambda_k(0)\}_{k=1}^\infty,$ using our notation. Our analysis above shows that the set 
$\lim_{\varepsilon\to0}S^\varepsilon$ has, in fact, a band-gap structure, 
with infinitely many gaps opening  in the interval $[0,\infty),$ as $\varepsilon\to0.$ This fact suggests possible applications of the above composite structures to the design of optical or acoustic band-gap materials, which we discuss in Section \ref{numericssection}. The above effect also raises a mathematical question of the analysis of the limit behaviour of the operators ${A}^\varepsilon$ in the case when $(a,b)$ is a bounded interval, which we study in the next section. 

In what follows we assume that 
$a,b\in\varepsilon F_1.$ Our results are also easily carried over to the case when $a,b\in\varepsilon F_0,$
if one modifies (\ref{bilinearform}) on those connected components of $\varepsilon F_0$ that contain $a$ or $b,$ by 
changing the related coefficient from $\varepsilon^2$ to unity.

\section{Spectral behaviour on a bounded interval}
\label{boundedintervalsection}

It is known that the classical, ``moderate-contrast'', analogue of the 
problem (\ref{ourproblem})--(\ref{bilinearform}) leads to limit spectra of different kinds 
for problems on bounded and unbounded intervals $(a,b):$ the limit set in the case of the 
problem in the whole space is purely absolutely continuous while in the case 
$-\infty<a<b<\infty$ it is purely discrete, {\it i.e.} it consists of eigenvalues with finite multiplicities, 
see {\it e.g.} \cite{BLP}. A similar situation occurs in 
multidimensional high-contrast problems where the inclusion $F_0\cap Q$ has a non-zero distance to the boundary 
of $Q,$ see \cite{Zh2000}, where, in addition, some eigenvalues of infinite multiplicity are present.

As we shall see next, this is not the case for the problem 
(\ref{ourproblem})--(\ref{bilinearform}), in particular rescaling $y=x/\varepsilon$ and 
replacing the form (\ref{bilinearform}) with an integral over the whole of ${\mathbb R},$ 
leads to higher-order errors in the limit as $\varepsilon\to0,$ which can be 
ignored in the leading-order, ``homogenised'', description of the operator 
${A}^\varepsilon.$ 

In this section we employ, for convenience, the following notation: $\Omega:=(a,b),$ 
$\Omega^\varepsilon:=\Omega\cap(\varepsilon F_1),$ $\Omega^\varepsilon:=
\Omega\cap(\varepsilon F_0).$

\subsection{The convergence result}
\label{rigorousstatementsection}

The following theorem holds.

\begin{theorem}
\label{maintheorembounded}
Consider an operator ${A}^\varepsilon$ from the class described in Section \ref{setupsubsection}, subject to 
the geometric modification mentioned at the end of Section \ref{wholespacesection}. The set $\lim_{\varepsilon\to0}S^\varepsilon$ is given by the union 
of solutions to the equation (\ref{limiteq}) for all $\theta\in[0,1).$  In particular, it is independent of the choice of the 
space ${\mathfrak H}$ in Section \ref{setupsubsection}.

\end{theorem}


The inclusion of the union of the spectra $\sigma(A(\theta))$ in the set $\lim_{\varepsilon\to0}S^\varepsilon$ is proved in the same way as in the case of the whole-space problem, see the proof of Theorem \ref{maintheoremwholespace}. In what follows we therefore discuss the  converse inclusion ({\it cf.} Theorem \ref{spcompthm}).


Notice first that for each $\lambda_\ep \in \spec{A_\ep}$, there exists $u_\ep \in {\mathfrak H}$, $\norm{u_\ep}{L^2(\Omega)}=1,$ such that 
\begin{equation}
\label{eq:bdspcom2}
\int_{\Omega^\ep_1} p\left(\tfrac{x}{\ep}\right)u_\ep'(x)\overline{\varphi'(x)}\mathrm{d}x
+\ep^2\int_{\Omega^\ep_0}p\left(\tfrac{x}{\ep}\right)u_\ep'(x)\overline{\varphi'(x)}\mathrm{d}x
=\lambda_\ep\int_{\Omega}u_\ep(x)\overline{\varphi(x)}\mathrm{d}x
\end{equation} 
for all $\varphi\in {\mathfrak H}.$ Setting $\varphi=u_\ep$ in \eqref{eq:bdspcom2} yields the estimates (``{\it a priori} bounds'' )
\begin{align}
\norm{u_\ep'}{L^2(\Omega^\ep_1)}\le C_{\rm B},\ \ \ \ \ \ \ \ep \norm{u_\ep'}{L^2(\Omega^\ep_0)}\le C_{\rm B}, \label{eq:bdspcom3}
\end{align}
where $C_{\rm B}>0$ is independent of $\varepsilon.$

For every bounded interval $D$ we denote $D_1:=D\cap F_1$ and introduce the function space ({\it cf.} 
(\ref{Vthetadef}), (\ref{V1def}), (\ref{VNdef}))
\[
V(D):=\{u\in H^1(D): p(y)u'(y)=0{\ for\ } y\in D_1\}.
\]
The following statement is central to the proof of Theorem \ref{maintheorembounded}.

\begin{proposition}
\label{prop0}
There exists a constant $C_\perp>0$ such that 
\[
\label{nondegener}
\bigl\Vert P_{V(D)^\perp}u\bigr\Vert_{H^1(D)}\le C_\perp\Vert u'\Vert_{L^2(D_1)}
\]
for any bounded interval $D\subset{\mathbb R}$ and any $u\in H^1(D).$ 
\end{proposition}
Henceforth $V(D)^\perp$ denotes the orthogonal complement of $V(D)$ in the space $H^1(D)$ equipped with the usual inner product, and $P_{V(D)^\perp}u$ denotes the orthogonal projection of the function $u$ onto 
$V(D)^\perp.$

Consider the norm 
\begin{equation}
\label{eq1}
\vert\vert\vert u\vert\vert\vert:=\Biggl(\left\vert\int_{Q_1}u(y)\mathrm{d}y\right\vert^2 + \int_Q\bigl\vert u'(y)\bigr\vert^2\mathrm{d}y\Biggr)^{1/2},\ \ \ \ u\in H^1(Q),
\end{equation}
which, in view of Lemma \ref{prop111} in Appendix B, is equivalent to the usual $H^1$-norm.  In the same 
appendix the above Proposition \ref{prop0} is shown to be equivalent to the following ``uniform in $\theta$'' Poincar\'{e}-type 
inequality. 

\begin{proposition}
\label{lem1}
There exists a constant $\tilde{C}>0,$ which depends on $\alpha$ and $\beta$ only,
such that for any $\theta\in[0,1)$
\begin{equation}
\label{lem1:eq1}
\vert\vert\vert w\vert\vert\vert^2\le\tilde{C}\int_{Q_1}\bigl\vert w'(y)\bigr\vert^2\mathrm{d}y, \quad\ \forall w \in V^\perp(\theta).
\end{equation}
\end{proposition}
We denote by $V^\perp(\theta)$ the orthogonal complement of $V(\theta)$ defined by (\ref{Vthetadef}) in the space $H_\theta^1(Q)$ equipped with the inner product in $H^1(Q)$ associated to the norm (\ref{eq1}).

\begin{proof}
We first note some properties of functions that belong to the space $V^\perp(\theta),$ which follow immediately from the characterisation of the space $V(\theta)$ given 
in the proof of Lemma \ref{contV}.

\begin{lemma}
\label{prop2}
Let $w \in V^\perp(\theta),$ then

(i) The equation $w''(y) = 0$ holds for $y\in Q_0$. In particular, the function $w$ is linear on the $Q_0$-component of the unit cell: $w(y) = (w(\beta)-w(\alpha))(\beta-\alpha)^{-1}(y - w(\alpha))+ w(\alpha)$ for $y \in Q_0$.

(ii) For $\theta \neq 0$ one has  $(\beta-\alpha)^{-1}(w(\beta)-w(\alpha))=(1-\exp(-2\pi{\rm i}\theta))^{-1}
(\alpha+(1-\beta)\exp(-2\pi{\rm i}\theta))\int_{Q_1} w(y)\mathrm{d}y.$

(iii) For $\theta = 0$ one has $\int_{Q_1}w(y)\mathrm{d}y=0.$


\end{lemma}

We now return to the proof of Proposition \ref{lem1}. We consider three different cases, depending on the location of the quasimomentum within the Floquet-Bloch cell $[0,1).$ Case I: $\theta = 0$. By \eqref{eq1} and Lemma \ref{prop2} (i), (iii), we find that
$$
\vert\vert\vert w\vert\vert\vert^2=\int_{Q_1}\vert w '(y) \vert^2\mathrm{d}y + \frac{\vert w(\beta) - w(\alpha) \vert^2}{\beta-\alpha}.
$$
Since $w(1) = w(0)$, we obtain the estimate
\begin{align*}
\vert\vert\vert w\vert\vert\vert^2 & \le\int_{Q_1}\vert w '(y)\vert^2\mathrm{d}y+\frac{2}{\beta-\alpha}\left(
\vert w(\beta)-w(1)\vert^2+\vert w(0)-w(\alpha)\vert^2\right)\\
& = \int_{Q_1} \vert w '(y) \vert^2\mathrm{d}y + \frac{2}{\beta-\alpha}\left(\left\vert \int^1_\beta w '(y)\mathrm{d}y  \right\vert^2 + \left\vert \int_0^\alpha w '(y)\mathrm{d}y \right\vert^2 \right) \\
& \le \left(1+\frac{2\vert Q_1 \vert}{\beta-\alpha} \right)\int_{Q_1} \vert w '(y) \vert^2\mathrm{d}y.
& 
\end{align*}

Case II: $\theta\in(0, \delta)\cup(1-\delta,1)$, where $0<\delta<1/2$ is to be chosen appropriately. By \eqref{eq1} and 
Lemma \ref{prop2} (i), (ii), we find that
\begin{flalign*}
\vert\vert\vert w\vert\vert\vert^2 & = \left\vert \int_{Q_1} w(y)\mathrm{d}y \right\vert^2 + \int_{Q_1} \left\vert w '(y)
\right\vert^2\mathrm{d}y+\frac{\left\vert w(\beta)-w(\alpha)\right\vert^2}{\beta-\alpha}
\\ & = \left( \frac{1}{\beta-\alpha}+\frac{1}{\vert d_\theta \vert^2} \right)\left\vert w(\beta)-w(\alpha) \right\vert^2 + \int_{Q_1}
\left\vert w'(y)\right\vert^2\mathrm{d}y,
\end{flalign*}
where $d_\theta:=(\beta-\alpha)(1-\exp(-2\pi{\rm i}\theta))^{-1}(\alpha+(1-\beta)\exp(-2\pi{\rm i}\theta))$. 
From the fact that $w(1)=\exp(2\pi{\rm i}\theta)w(0)$ we infer
\begin{flalign*}
\left\vert w(\beta)-w(\alpha)\right\vert^2 & \le3\left(\left\vert w(1) - w(\beta) \right\vert^2 + \left\vert w(\alpha) - w(0) \right\vert^2 + \left\vert w(1)-w(0)\right\vert^2\right)\\
& \le3\left\vert Q_1 \right\vert \int_{Q_1} \vert w'(y)\vert^2\mathrm{d}y+3\left\vert\exp(2\pi{\rm i}\theta)-1\right\vert^2 \vert w(0) \vert^2,
\end{flalign*}
 and therefore 
\[
\vert\vert\vert w\vert\vert\vert^2\le\left(3\vert Q_1\vert+ \frac{1}{\beta-\alpha} +\frac{1}{\vert d_\theta \vert^2} \right)\int_{Q_1} \left\vert w '(y) \right\vert^2\mathrm{d}y\ \ \ \ \ \ \ \ \ \ \ \ \ \ \ \ \ \ \ \ 
\]
\[
\ \ \ \ \ \ \ \ \ \ \ \ \ \ \ \ \ \ \ +3\left(\frac{1}{\beta-\alpha} + \frac{1}{\vert d_\theta \vert^2} \right)\left\vert\exp(2\pi{\rm i}\theta)-1\right\vert^2 \vert w(0)\vert^2.
\]
Notice that $\vert d_\theta \vert^2 =(\beta-\alpha)^2(2-2\cos(2\pi \theta))^{-1}(\alpha^2 + (1-\beta)^2 + 2\alpha(1-\beta)\cos(2\pi\theta))$, hence
$\vert d_\theta \vert$ vanishes at $\theta=1/2$ for the special case $\alpha=1-\beta$. In view of this observation and in order to have a bound on the constant $d_\theta$ we require that $\delta<1/4.$ 
Further, by continuity of the embedding of $H^1(Q)$ in $C(\overline{Q}),$ there exists a constant $\hat{c},$ which is independent of $\theta,$ such that
$$
\bigl\vert w(0)\bigr\vert\le\hat{c}\vert\vert\vert w\vert\vert\vert,
$$
and thus
\[
\vert\vert\vert w\vert\vert\vert^2\le \left(3\vert Q_1\vert+ \frac{1}{\beta-\alpha} + \frac{1}{\vert d_\theta \vert^2} \right)\int_{Q_1} \left\vert w ' (y)\right\vert^2\mathrm{d}y\ \ \ \ \ \ \ \ \ \ \ \ \ \ \ \ \ \ \ 
\]
\beq
\ \ \ \ \ \ \ \ \ \ \ \ \ \ \ \ \ \ \ +3\left(\frac{1}{\beta-\alpha}+\frac{1}{\vert d_\theta \vert^2} \right)\left\vert\exp(2\pi{\rm i}\theta)-1\right\vert^2\hat{c}^2\vert\vert\vert w\vert\vert\vert^2.
\eeq{normwineq}
We now choose $\delta<1/4$ so that  $\left((\beta-\alpha)^{-1}+\vert d_\theta\vert^{-2} \right)\left\vert\exp(2\pi{\rm i}\theta)-1\right\vert^2\hat{c}^2<1/2,$ and hence $\vert d_\theta \vert^{-2}$ is bounded above by a constant independent of $\theta.$ The inequality (\ref{normwineq}) now immediately implies the required estimate.
 
Case III: $\theta\in[\delta, 1-\delta].$  For given $x \in (\beta,1], y \in [0,\alpha)$ we write
\begin{gather*}
w(x) = \int^x_\beta w'(t)\mathrm{d}t + w(\beta), \quad \quad
w(y) = - \int^\alpha_y w'(t)\mathrm{d}t + w(\alpha),
\end{gather*}
which implies, in view of Proposition \ref{prop2} (ii),
\begin{align*}
w(x) - w(y)  & = w(\beta) - w(\alpha) + \left( \int^x_\beta + \int^\alpha_y \right) w'(t)\mathrm{d}t \\
& = d_\theta \int_{Q_1} w(y)\mathrm{d}y + \left( \int^x_\beta + \int^\alpha_y \right) w'(t)\mathrm{d}t.
\end{align*}
In particular, substituting $x=1$, $y = 0$ and using the fact that $w(1) = \exp(2\pi{\rm i}\theta)w(0)$, we obtain
\begin{equation*}
w(1) = \frac{d_\theta}{1 - \exp(-2\pi{\rm i}\theta)}\int_{Q_1} w(y)\mathrm{d}y + \frac{1}{1 - \exp(-2\pi{\rm i}\theta)}\int_{Q_1}w '(y)\mathrm{d}y ,
\end{equation*}
whence
\[
w(x) = - \int^1_x w'(t)\mathrm{d}t + w(1) =   - \int^1_x w'(t)\mathrm{d}t\ \ \ \ \ \ \ \ \ \ \ \ \ \ \ \ \ \ \ \ \ \ \ \ \ \ \ \ \ \ \ \ \ \ \ \ \ \ \ \ \ \ \ \ \ \ \ \ \ \ \ \ \ \ 
\]
\[
\ \ \ \ \ \ \ \ \ \ \ \ \ \ \ \ \ \ \ \ \ \ \ \ \ \ \ \ \ \ \ \ \ \ +\frac{d_\theta}{1 - \exp(-2\pi{\rm i}\theta)}\int_{Q_1} w(y)\mathrm{d}y
+\frac{1}{1 - \exp(-2\pi{\rm i}\theta)}\int_{Q_1} w'(y)\mathrm{d}y.
\]
Integrating the last identity over $(\beta,1]$ yields
\[
\int^1_\beta w(y)\mathrm{d}y = -\int^1_\beta \left( \int^1_x w'(t)\mathrm{d}t\right) \mathrm{d}x\ \ \ \ \ \ \ \ \ \ \ \ \ \ \ \ \ \ \ \ \ \ \ \ \ \ \ \ \ \ \ \ \ \ \ \ \ \ \ \  \ \ \ \ \ \ \ \ \ \ \ \ \ \ \ \ \ \ \ 
\]
\beq
\ \ \ \ \ \ \ \ \ \ \ \ \ \ \ \ \ \ \ \ \ \ \ \ \ \ \ \ +\frac{(1-\beta) d_\theta}{1 - \exp(-2\pi{\rm i}\theta)}\int_{Q_1} w(y)\mathrm{d}y + \frac{1-\beta}{1 - \exp(-2\pi{\rm i}\theta)}\int_{Q_1}w'(y)\mathrm{d}y, 
\eeq{ee1}
Similarly, we write
\begin{align*}
w(y)  & = \int^y_0 w '(t)\mathrm{d}t +\exp(-2\pi{\rm i}\theta)w(1) \\
& = \int^y_0 w '(t)\mathrm{d}t + \frac{\exp(-2\pi{\rm i}\theta)d_\theta}{1 - \exp(-2\pi{\rm i}\theta)}
\int_{Q_1} w(y)\mathrm{d}y + 
\frac{\exp(-2\pi{\rm i}\theta)}{1 - \exp(-2\pi{\rm i}\theta)}\int_{Q_1} w'(y)\mathrm{d}y,  
\end{align*}
which upon integration over $(0,\alpha)$ yields
\[
\int_0^\alpha w(y)\mathrm{d}y=\int_0^\alpha\left( \int^y_0 w'(t)\mathrm{d}t\right)\mathrm{d}y\ \ \ \ \ \ \ \ \ \ \ \ \ \ \ \ \ \ \ \ \ \ \ \ \ \ \ \ \ \ \ \ \ \ \ \ \ \ \ \ \ \ \ \ \ \ \ \ \ \ \ \ \ \ 
\]
\beq
\ \ \ \ \ \ \ \ \ \ \ \ \ +\frac{\alpha\exp(-2\pi{\rm i}\theta) d_\theta}{1-\exp(-2\pi{\rm i}\theta)}\int_{Q_1}w(y)\mathrm{d}y+
\frac{\alpha\exp(-2\pi{\rm i}\theta)}{1-\exp(-2\pi{\rm i}\theta)}\int_{Q_1} w'(y)\mathrm{d}y.
\eeq{ee2}
Combining equations \eqref{ee1} and \eqref{ee2} we obtain
\[
\left(1-\frac{\left(1-\beta+\alpha\exp(-2\pi{\rm i}\theta)\right)d_\theta}{1-\exp(-2\pi{\rm i}\theta)} \right)
\int_{Q_1}w(y)\mathrm{d}y=\int_0^\alpha\left( \int^y_0 w'(t)\mathrm{d}t\right)\mathrm{d}y\ \ \ \ \ \ \ \ \ \ \ \ 
\]
\[
\ \ \ \ \ \ \ \ \ \ \ \ \ -\int^1_\beta\left(\int^1_x w'(t)\mathrm{d}t\right) \mathrm{d}x+\frac{\left(1-\beta+\alpha\exp(-2\pi{\rm i}\theta)\right)}{1-\exp(-2\pi{\rm i}\theta)}\int_{Q_1}w'(y)\mathrm{d}y.
\]
Squaring both sides and using the Cauchy-Schwarz inequality yields 
\[
\left\vert 1-\frac{\left(1-\beta+\alpha\exp(-2\pi{\rm i}\theta) \right)d_\theta}{1-\exp(-2\pi{\rm i}\theta)} \right\vert^2 \left\vert \int_{Q_1} w(y)\mathrm{d}y\right\vert^2\ \ \ \ \ \ \ \ \ \ \ \ \ \ \ \ \ \ \ \ \
\]
\[
\ \ \ \ \ \ \ \ \ \ \ \ \ \ \ \ \ \ \ \ \ \ \ \ \ \ \ \ \  \ \ \le 2\left(4+\frac{\left\vert 1-\beta+\alpha \exp(-2\pi{\rm i}\theta) \right\vert^2}{\vert 1-\exp(-2\pi{\rm i}\theta) \vert^2}  \right) \int_{Q_1} \vert w '(y) \vert^2\mathrm{d}y.
\]
A direct calculation shows that the coefficient in the left-hand side of the last inequality is separated from zero in 
the range of $\theta$ considered. 

Finally, we argue that 
\[
\vert w(\beta)-w(\alpha) \vert \le4\bigl(\vert w(\beta)-w(1)\vert+\vert w(0)-w(\alpha)\vert+\vert w(0)\vert +\vert w(1)\vert\bigr)
\]
\[
\ \ \ \ \ \ \ \ \ \ \ \ \ \ \ \ \ \ \ =4\left(\biggl\vert\int_\beta^1w'(y)\mathrm{d}y\biggr\vert+\biggl\vert\int_0^\alpha 
w'(y)\mathrm{d}y\biggr\vert+\vert w(0)\vert +\vert w(1)\vert\right),
\]
and 
\[
(1-\beta)w(1) =\int_\beta^1\int^1_x w '(t)\mathrm{d}t\mathrm{d}x + \int^1_\beta w(x)\mathrm{d}x,
\]
\[
\alpha w(0) = -\int_0^\alpha\int^x_0 w '(t)\mathrm{d}t\mathrm{d}x+\int_0^\alpha w(x)\mathrm{d}x.
\]
The required inequality follows, since by \eqref{eq1} and Lemma \ref{prop2} (i), 
\[
\vert\vert\vert w\vert\vert\vert=\left\vert \int_{Q_1} w(y)\mathrm{d}y\right\vert^2 + \frac{\left\vert w(\beta) - w(\alpha) \right\vert^2}{\beta-\alpha} + \int_{Q_1} \vert w '(y) \vert^2\mathrm{d}y. 
\]
This completes the proof of Proposition \ref{lem1}.
\end{proof}

We now resume the proof of Theorem \ref{maintheorembounded}.  

Let $N_\varepsilon$ be the smallest integer such that $\overline{\Omega}\subset \varepsilon[a/\varepsilon]+\cup_{n=0}^{N_\varepsilon-1}\ep (n,n+1)=:\tilde{\Omega}^\varepsilon.$ 
For each $\varepsilon,$ we denote $D^\varepsilon:=\varepsilon^{-1}\tilde{\Omega}^\varepsilon$ and use the 
rescaling operator  $T_\varepsilon$ defined by 
$(T_\varepsilon u)(y)=u(\varepsilon y),$ $y\in D^\varepsilon,$ and an extension operator $E_\varepsilon$ from 
${\mathfrak H}$ to  $H^1(\tilde{\Omega}^\varepsilon)$ such that 
\beq
\Vert E(u_\varepsilon)\Vert_{L^2(\tilde{\Omega}^\varepsilon)}\le(1+C_{\rm E}\varepsilon)\Vert u_\varepsilon\Vert_{L^2(\Omega)}\,,\ \ \Vert E(u_\varepsilon)'\Vert_{L^2(\tilde{\Omega}^\varepsilon)}\le(1+C_{\rm E}\varepsilon)\Vert u'_\varepsilon\Vert_{L^2(\Omega)}
\eeq{extest}
for all $\varepsilon>0,$ where the constant $C_{\rm E}>0$ is the same for all $\varepsilon.$ (It is clear, for example, 
that  the continuous extension of $u_\varepsilon$ by constants satisfies this requirement.)
We then define functions $U_\varepsilon$ by the formula 
$U_\ep:=T_{\varepsilon}^{-1}P_{V(D^\varepsilon)} T_\varepsilon E(u_\varepsilon)$ for each value of $\varepsilon.$ 

Extending each of the functions $U_\varepsilon$ to the whole of ${\mathbb R}$ with period $\varepsilon N_\varepsilon,$
we consider the ``discrete Floquet-Bloch transform'' of the functions $U_\varepsilon$ as follows
\[
U^j_\ep(x)=\frac{1}{N_\varepsilon}\sum_{k=0}^{N_\varepsilon-1}U_\varepsilon(x+\varepsilon k)\exp(2\pi{\rm i}jk/N_\varepsilon),\ \ \ \ x\in\tilde{\Omega}_\varepsilon,\ \ \ \ j=0,1,..., N_\varepsilon-1,
\]
so that  $T_\varepsilon U^j_\ep\in V(j/N_\varepsilon)$ and $U_\ep = \sum_{j=0}^{N_\varepsilon-1} U^j_\ep.$ Using the fact that the eigenfunctions $v^k(j/N_\varepsilon),$ $k\in{\mathbb N},$
form a complete system in the $L^2(Q)$-closure of the set $V(j/N_\varepsilon)$, we represent 
$U^j_\varepsilon$ in terms of them so that 
\beq
U_\ep(x) = \sum_{j=0}^{N_\varepsilon-1} \sum_{k=1}^\infty\hat{U}^k_\ep\bigl(\tfrac{j}{N_\varepsilon}\bigr)
v^k\bigl(\tfrac{j}{N_\varepsilon},\tfrac{x}{\ep} \bigr),
\eeq{eq:1}
\beq
U'_\ep(x)=\varepsilon^{-1}\sum_{j=0}^{N_\varepsilon-1} \sum_{k=1}^\infty
\hat{U}^k_\ep\bigl(\tfrac{j}{N_\varepsilon}\bigr)(v^k)'\bigl(\tfrac{j}{N_\varepsilon},\tfrac{x}{\ep}\bigr), 
\eeq{UprimeParseval}
where $\hat{U}^k_\ep(j/N_\varepsilon)\in{\mathbb C}.$ Applying the Parseval identity to (\ref{eq:1}), followed by Proposition \ref{prop0} and the first of the {\it a priori} bounds (\ref{eq:bdspcom3}) yields 
\[
\varepsilon N_\varepsilon\sum_{j = 0}^{N_\varepsilon-1} \sum_{k=1}^\infty \Bigl\vert \hat{U}^k_\ep\bigl(\tfrac{j}{N_\varepsilon}\bigr)\Bigr\vert^2=\norm{U_\ep}{L^2(\tilde{\Omega}_\varepsilon)}^2\ge(1-C_\perp C_{\rm B}\varepsilon)\norm{u_\ep}{L^2(\Omega)}^2.
\]
The last estimate, in combination with the first inequality in (\ref{extest}) and the fact that $\Vert u_\varepsilon\Vert_{L^2(\Omega)}=1,$ implies
\beq
1-C_\perp C_{\rm B}\varepsilon\le\varepsilon N_\varepsilon\sum_{j = 0}^{N_\varepsilon-1} \sum_{k=1}^\infty \Bigl\vert \hat{U}^k_\ep\bigl(\tfrac{j}{N_\varepsilon}\bigr)\Bigr\vert^2\le1+C_{\rm E}\varepsilon
\eeq{eq:2}
Denoting by $\delta(\cdot-\theta)$ the Dirac mass at $\theta$, we infer from (\ref{eq:2}) the existence of $C>0$ such that
\beq
\biggl\vert\sum_{k=1}^\infty \int_0^1\mathrm{d}\mu^k_\ep-1\biggr\vert\le C\varepsilon,\ \ \ \ \ 
{\rm d}\mu^k_\ep(\theta):=\varepsilon N_\varepsilon\sum_{j=0}^{N_\varepsilon-1}\Bigl\vert \hat{U}^k_\ep\bigl(\tfrac{j}{N_\varepsilon}\bigr)\Bigr\vert^2 \delta\bigl(\theta-\tfrac{j}{N_\varepsilon}\bigr){\rm d}\theta.
\eeq{eq:5}
Clearly, for each $k$ the sequence $\{\mu^k_\ep\}_\varepsilon$  is bounded in the space of Radon measures on $[0,1].$ Therefore, up to a subsequence,  $\mu^k_\ep$ weakly converge as $\varepsilon\to0$ to some measure $\mu^k:$  
\begin{equation}
\label{eq:4}
\int_0^1u\ \mathrm{d}\mu^k_\ep \rightarrow \int_0^1u\ \mathrm{d}\mu^k, \quad \forall u \in C[0,1].
\end{equation}
The above result follows from recalling that the space of finite Radon measures on $[0,1]$ coincides with the dual space $C[0,1]^\star$ and hence bounded sets of Radon 
measures  and relatively compact with respect to weak star convergence in this space. 

Furthermore, taking into account (\ref{UprimeParseval}), the second of the estimates (\ref{extest}) and the second of the {\it a priori} bounds (\ref{eq:bdspcom3}), we write
\beq
C_{\rm B}(1+C_{\rm E}\varepsilon)\ge\varepsilon^2\Vert U_\varepsilon'\Vert_{L^2(\tilde{\Omega}_\varepsilon)}^2=
\varepsilon N_\varepsilon\sum_{j=0}^{N_\varepsilon-1} \sum_{k=1}^\infty\Bigl\vert\hat{U}^k_\ep\bigl(\tfrac{j}{N_\varepsilon}\bigr)
\Bigr\vert^2\bigl\Vert(v^k)'\Bigl(\tfrac{j}{N_\varepsilon},\cdot\bigr)\Bigr\Vert^2_{L^2(Q)}
\eeq{Blochest0}
\beq
\ge\Vert p^{-1}\Vert_{L^\infty(Q)}\varepsilon N_\varepsilon\sum_{j=0}^{N_\varepsilon-1}\sum_{k=1}^\infty
\Bigl\vert\hat{U}^k_\ep\bigl(\tfrac{j}{N_\varepsilon}\bigr)
\Bigr\vert^2\lambda_k\bigl(\tfrac{j}{N_\varepsilon}\bigr)
\eeq{Blochest1}
\beq
=\Vert p^{-1}\Vert_{L^\infty(Q)}\sum_{k=1}^\infty\int_0^1\lambda_k(\theta){\rm d}\mu^k_\ep(\theta).
\eeq{Blochest2}
In order to obtain (\ref{Blochest1}), we use the fact that $v^k(\theta),$ $k\in{\mathbb N},$ are 
eigenfunctions of the operator $A(\theta)$ with eigenvalues $\lambda_k(\theta),$ and (\ref{Blochest2}) 
is a version of the same expression written in terms of the measures defined in (\ref{eq:5}). For any integer $K\ge 2$ the inequality (\ref{Blochest0})--(\ref{Blochest2}) immediately implies
\[
\sum_{k=K}^\infty\int_0^1{\rm d}\mu^k_\ep(\theta)\le \Vert p^{-1}\Vert_{L^\infty(Q)}^{-1}C_{\rm B}(1+C_{\rm E}\varepsilon)\min_{\theta\in[0,1)}\lambda_K(\theta)^{-1}\stackrel{K\to\infty}{\longrightarrow}0.
\]
In view of  (\ref{eq:5}) and (\ref{eq:4}) we thus argue that $\sum_{k=1}^\infty\mu^k([0,1])=1,$ 
hence $\mu^k$ is a non-zero measure for some $k.$

We next show that $\lambda = \lambda_{k_0}(\theta)$ for some $k_0,\theta$. 
To this end we set $\varphi=U_\varepsilon$ in the weak formulation \eqref{eq:bdspcom2}:
\begin{equation}
\label{weakutilde}
\int_{\Omega^\ep_1}p\left(\tfrac{x}{\ep} \right)u_\ep'(x)\overline{U_\varepsilon'(x)}\mathrm{d}x+
\ep^2 \int_{\Omega^\ep_0}p\left(\tfrac{x}{\ep} \right)u_\ep'(x)\overline{U_\varepsilon'(x)}\mathrm{d}x=\lambda_\ep \int_{\Omega}u_\ep(x)\overline{U_\varepsilon(x)}\mathrm{d}x. 
\end{equation}
Notice that, by the definition of the space $V(D^\varepsilon),$ the first term in the equation (\ref{weakutilde}) vanishes.
We also have, in view of Proposition \ref{prop0} and the second of the estimates \eqref{eq:bdspcom3}, as 
$\varepsilon\to0:$
\[
\ep^2\int_{\Omega^\ep_0}p\left(\tfrac{x}{\ep}\right)u_\ep'(x)\overline{U_\ep'(x)}\mathrm{d}x=
\ep^2\int_{\Omega^\ep_0}p\left(\tfrac{x}{\ep}\right)U_\ep'(x)\overline{U_\ep'(x)}\mathrm{d}x+o(1)
\]
\[
=\sum_{k=1}^{\infty}\sum_{l=1}^{\infty}\sum_{j=0}^{N_\varepsilon-1}\int_{\ep Q_0}p\left(\tfrac{x}{\ep}\right)\hat{U}^k_\ep\bigl(\tfrac{j}{N_\varepsilon}\bigr) (v^{k})'\bigl(\tfrac{j}{N_\varepsilon},\tfrac{x}{\ep}\bigr)\overline{\hat{U}^l_\ep\bigl(\tfrac{j}{N_\varepsilon} \bigr)(v^{l})'\bigl(\tfrac{j}{N_\varepsilon},\tfrac{x}{\ep}\bigr)}\mathrm{d}x +o(1)
\]
\begin{equation}
\label{eq:bdspcom6}
=\varepsilon N_\varepsilon\sum_{k=1}^{\infty} \sum_{j=0}^{N_\varepsilon-1} \lambda_k\bigl(\tfrac{j}{N_\varepsilon}\bigr)\Bigl\vert\hat{U}^k_\ep\bigl(\tfrac{j}{N_\varepsilon}\bigr)\Bigr\vert^2+o(1)
= \sum_{k=1}^{\infty} \int_0^1\lambda_k(\theta)\mathrm{d}\mu_\ep^k+o(1).
\end{equation}
In the last three equalities of (\ref{eq:bdspcom6}) we use \eqref{eq:1}, \eqref{eq:5} and the fact that $v^k(\theta)$ are orthogonal eigenfunctions of 
$A(\theta)$ with eigenvalues $\lambda_k(\theta)$. Similarly we have
\beq
\int_0^1 u_\varepsilon(x)\overline{U_\varepsilon(x)}\mathrm{d}x=\sum_{k=1}^\infty\mathrm{d}\mu_\varepsilon^k
+o(1),
\eeq{rhslimit}
as $\varepsilon\to0.$
Combining (\ref{weakutilde}), \eqref{eq:bdspcom6}, (\ref{rhslimit}) yields
\begin{equation}
\label{eq:bdspcom11}
\sum_{k=1}^{\infty} \int_0^1\lambda_k(\theta)\mathrm{d}\mu_\ep^k(\theta)
= \lambda_\ep \sum_{k=1}^{\infty} \int_0^1\mathrm{d}\mu_\ep^k +o(1).
\end{equation} 
Finally, passing to the limit $\ep \rightarrow 0$ in \eqref{eq:bdspcom11} and using the fact $\lambda_k$ are continuous functions of  $\theta$ (see Appendix C) yields
$$
\sum_{k=1}^{\infty} \int_0^1\lambda_k(\theta)\mathrm{d}\mu^k(\theta)=\lambda\sum_{k=1}^{\infty} \int_0^1
\mathrm{d}\mu^k.
$$
We already know that $\mu^{k_0}$ is a non-zero measure for some $k_0,$ hence $\lambda = \lambda_{k_0}(\theta)$ for some $\theta$.

\section{A modified problem with a compact perturbation, and the associated defect modes}\label{numericssection}

\subsection{Analytical setup}

In this section we discuss a modified version of the setup of Section \ref{setupsubsection}, as follows.
Consider the operator $\tilde{A}^\varepsilon$ defined by the bilinear form
\beq
(\tilde{A}^\varepsilon u,v)=\int_a^b\Bigl(p_{\rm d}\chi_{\rm d}(x)+p(x/\varepsilon)\bigl(\varepsilon^2\chi_0(x/\varepsilon)+\chi_1(x/\varepsilon)\bigr)\bigl(1-\chi_{\rm d}(x)\bigr)\Bigr)u'(x)v'(x)dx,
\eeq{bilinearformcompact}
$u,v\in {\mathfrak H}.$ Here the space ${\mathfrak H}$ is as before, and $\chi_{\rm d}$ is the indicator function of a ``defect'' interval 
$I_{\rm d}$ whose closure is assumed to be contained in $(a,b),$ and 
$p_{\rm d}$ is the corresponding ``defect'' coefficient (or ``defect strength'').  Analogously to the above, we assume that $a,b,$ and the end-points of $I_{\rm d}$ belong to the set $\varepsilon F_1,$ otherwise we modify (\ref{bilinearformcompact}) 
on those connected components of $\varepsilon F_0$ that contain $a,$ $b$ and the end-points of 
$I_{\rm d}$ by changing the related coefficient from $\varepsilon^2$ to the unity. We denote by $\tilde{S}^\varepsilon$ the spectrum of the operator ${A}^\varepsilon.$

A formal two-scale asymptotic procedure carried out on the equation ({\it cf.} (\ref{ourproblem})) $\tilde{A}^\varepsilon u=\lambda u$
suggests that:

1) The set $\lim_{\varepsilon\to0}\tilde{S}^\varepsilon$ is independent of the choice of the space ${\mathfrak H}$ and is given by the union 
of solutions to the equation (\ref{limiteq}) for all $\theta\in[0,1)$ and a sequence of ``defect eigenvalues''  
$\{p_{\rm d}\pi^2j^2/\vert I_{\rm d}\vert^2\}_{j=1}^\infty;$

2) The ``defect eigenfunctions'' $u_j$ corresponding to the above eigenvalues, $j\in{\mathbb N},$ decay exponentially away from the boundary of the defect: 
\[
\bigl\vert u_j(x)\bigr\vert\le C\exp\bigl(-\varepsilon^{-1}{\rm dist}\{x, I_{\rm d}\}\bigr), \ \ \ \ x\in(a,b)\setminus I_{\rm d},\ \ \ \ C>0.
\]
We next present numerical evidence that supports these claims.

\subsection{Numerical results for the modified problem}

We consider a defect of length $\vert I_{\rm d}\vert=1/2$ and strength $p_{\rm d}=2$ in the middle of the interval $(a,b)=(0,1),$ {\it i.e.} $I_{\rm d}=(1/4, 3/4).$ For each value of $\varepsilon$ such that $N:=\varepsilon^{-1}$ is a positive integer, we describe the intervals $(0,1/4)$ and $(3/4,1)$ on either sides of the defect according to 
the equation (\ref{bilinearformcompact}): each of them consists of $N$ cells of the same length $1/(4N),$ and in one half of each cell the coefficient in the form (the ``modulus'') takes the value
$1/(4N)^2$ while in the other half it is equal to unity. Importantly, we assume that the endpoints both of the 
interval $(a,b)$ and of the defect are belong to the material component where the modulus is equal to unity.



The results of solving the above problems with finite elements are given in Tables \ref{tabledirichlet} 
and \ref{tableneumann}.  The values for the trapped mode are in good agreement with the values obtained by the asymptotic method: $\lambda_\star^{(2)}=78.9568,$ $\lambda_\star^{(3)}=315.8273,$ 
$\lambda_\star^{(5)}=710.6115 ,$ $\lambda_\star^{(6)}=1263.3094,$ where the superscript is the number of 
the stop band containing the related eigenvalue.


In addition, the profiles obtained for such trapped modes (see Figure \ref{crasfig2} for the case of 
periodic boundary conditions) suggest that the 
number of half-oscillations in a trapped mode is equal to the number of the mode in the sequence, 
which resembles the behaviour of the usual Neumann 
eigenfunctions on the defect. We also note that the decay of the trapped modes appears to be
exponential, as can be seen in Figure \ref{crasfig3}: the larger the contrast
(and hence the number of subdivisions of the string) the more localised the mode,
irrespective of the boundary conditions at the endpoints of the string.

\begin{table}[h]
\begin{tabular}{| c | c | c |c | c | c|}
\hline
\multicolumn{6}{| c |}{Dirichlet boundary conditions}\\
\hline $\lambda^{(k,128)}_{\rm min}$  &  $\lambda^{(k,128)}_{\rm max}$ & $\lambda^{(k,128)}_\star$ & $\lambda^{(k,256)}_\star$ & $\lambda^{(k,512)}_\star$ & $\lambda^{(k,1024)}_\star$ \\
\hline
11.7939	&	39.4603	&	- & - &	- & -\\
\hline
65.7875	&	157.8859	&	75.7674 & 77.2502 &	78.0741 & 78.7304\\
\hline
187.6799	&	355.2599	&	293.9534 & 304.1141 &309.7163 & 314.2461\\
\hline
386.1413	&	622.2747	&	- & - & - & -\\
\hline
662.9213	&	986.7685	&	682.6577 & 694.4984 & 702.0486 & 708.4576\\
\hline
1018.4394	&	1421.0468	&	1225.1298 & 1232.2190 & 1243.1182 & 1258.2799\\
\hline
\end{tabular}

\caption{Stop bands and trapped modes for the modified problem with a defect, subject to the Dirichlet boundary conditions:
$\lambda^{(k,128)}_{\rm min}$ and $\lambda^{(k,128)}_{\rm max}$ are the lower and upper bounds of the $k^{th}$ stop band for $N=128$, and
$\lambda^{(k,128)}_\star$, $\lambda^{(k,256)}_\star$, $\lambda^{(k,512)}_\star$,
$\lambda^{(k,1024)}_\star$ are the trapped-mode eigenvalues in the $k^{th}$ stop band, evaluated for $N=128$, $N=256$, $N=512$ and $N=1024$ respectively.
}
\label{tabledirichlet}
\end{table}

\begin{table}[h]
\begin{tabular}{| c | c | c |c | c |c |}
\hline
\multicolumn{6}{| c |}{Neumann boundary conditions}\\
\hline $\lambda^{(k,128)}_{\rm min}$  &  $\lambda^{(k,128)}_{\rm max}$ & $\lambda^{(k,128)}_\star$ & $\lambda^{(k,256)}_\star$ & $\lambda^{(k,512)}_\star$ & $\lambda^{(k,1024)}_\star$\\
\hline
11.7515	&	39.4980	&	- & - &	- & -\\
\hline
65.8359	&	157.8901	&	75.7676 & 77.2509 & 78.0741 & 78.7314 \\
\hline
187.7334	&	355.2765	&	293.9539 & 304.1145 & 309.7164 & 314.3057 \\
\hline
386.2091	&	622.2747	&	- & - & - & - \\
\hline
662.9779	&	986.7698	&	682.6578 & 694.4985 & 702.0486 & 708.6496 \\
\hline
1018.5163	&	1421.0556	&         1225.1298 & 1232.2190 & 1243.1193 & 1258.2626 \\
\hline
\end{tabular}
\caption{Stop bands and trapped modes for the modified problem with a defect, subject to the Neumann boundary conditions:
$\lambda^{(k,128)}_{\rm min}$ and $\lambda^{(k,128)}_{\rm max}$ are the lower and upper 
bounds of the $k^{th}$ stop band for $N=128$, and $\lambda^{(k,128)}_\star$, 
$\lambda^{(k,256)}_\star$, $\lambda^{(k,512)}_\star$,
$\lambda^{(k,1024)}_\star$ are the trapped-mode eigenvalues in the $k^{th}$ stop band, evaluated for 
$N=128$, $N=256$, $N=512$ and $N=1024$ respectively.
}
\label{tableneumann}
\end{table}

\begin{center}
\begin{figure}[h]
\centering
\includegraphics[scale=0.3]{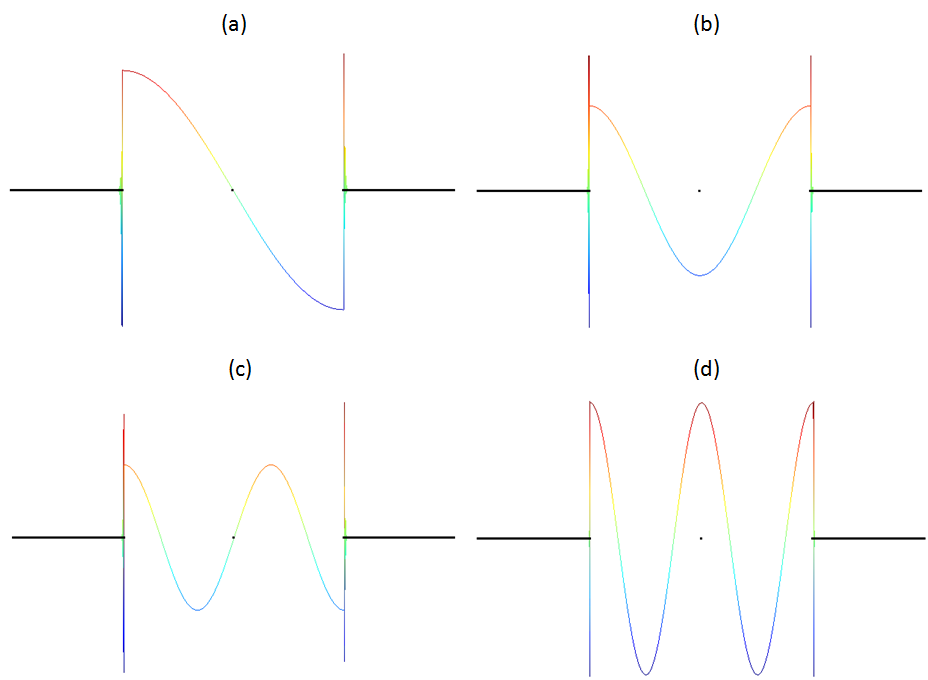}
\caption{Trapped eigenmodes for the modified problem with a defect and periodic boundary conditions
(N=1024), corresponding to eigenfrequencies: (a) $\lambda^{(2,1024)}_\star=78.07,$ 
(b) $\lambda^{(3,1024)}_\star=314.24,$ (c) $\lambda^{(5,1024)}_\star=708.46,$ 
(d) $\lambda^{(6,1024)}_\star=1258.28.$}
\label{crasfig2} 
\end{figure}
\end{center}

\begin{center}
\begin{figure}[h]
\centering
\includegraphics[scale=0.3]{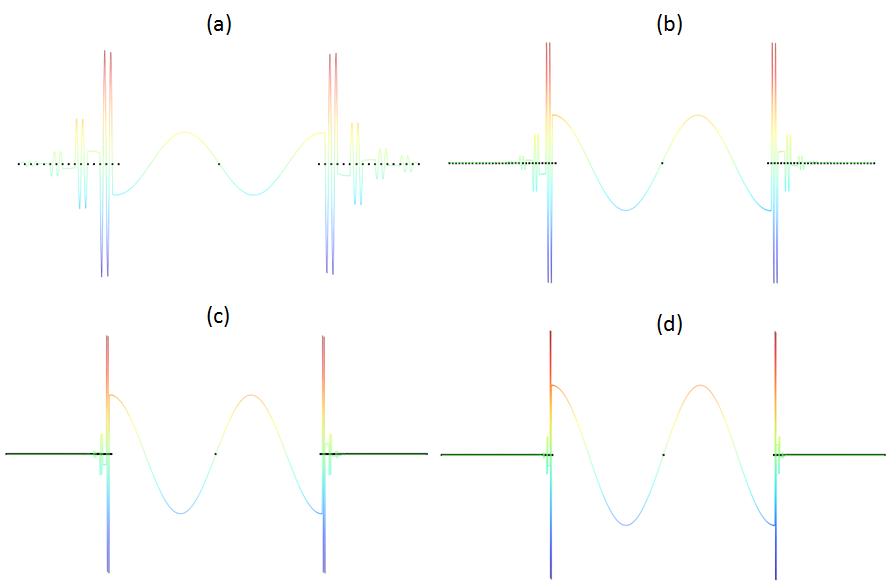}
\caption{Decay of the third trapped eigenmode (located in the fifth stop band) for the modified problem with 
a defect and periodic boundary conditions, as a function of contrast: (a) $\lambda^{(5,32)}_\star=668.86,$ (b) $\lambda^{(5,64)}_\star=682.89,$ (c) $\lambda^{(5,128)}_\star=694.50,$ (d) $\lambda^{(5,256)}_\star=702.68.$
}
\label{crasfig3} 
\end{figure}
\end{center}

\subsection{Photonic band gaps and trapped modes in high-contrast multi-layered dielectric structures}\label{numericssectionb}


The string problem emerges, among other contexts, in the study of wave propagation in 
one-dimensional photonic crystals {\it i.e.} multi-layered dielectric structures invariant 
along two directions. In what follows we set these directions to be $x_1$ and $x_3$ in the usual 
Euclidean representation $x=(x_1, x_2, x_3). $ 

We consider those solutions $({\mathcal E}, {\mathcal H})$ to the 
classical system of Maxwell equations (\cite{Jackson}) that have the form
\[
{\mathcal E}(x_1,x_2,x_3,t)=E(x_2)\exp(i(\kappa x_3-\omega t)),\ \ \ 
{\mathcal H}(x_1,x_2,x_3,t)=H(x_2)\exp(i(\kappa x_3-\omega t)),
\]
where $t$ is time, $\omega$ is the angular frequency, and $\kappa\geq 0$ is a 
``propagation constant''. We write the Maxwell equations for the field variable $(E,H):$
\begin{equation}
\left\{
\begin{array}{cc}
E_3'-{\rm i}\kappa E_2& = {\rm i}\omega\mu H_1\nonumber, \\
{\rm i}\kappa E_1& = {\rm i}\omega \mu H_2 \nonumber, \\
-E_1'& = {\rm i}\omega \mu H_3 \nonumber, \\
\end{array}
\right.
\ \ \ \ 
\left\{
\begin{array}{cc}
-H_3'+{\rm i}\kappa H_2& = {\rm i}\omega\varepsilon E_1\nonumber, \\
-{\rm i}\kappa H_1& = {\rm i}\omega\varepsilon E_2 \nonumber, \\
H_1'& = {\rm i}\omega \varepsilon E_3 \nonumber, \\
\end{array}
\right.
\end{equation}
Here $\mu$ is the magnetic permeability, $\varepsilon$ is the electric 
permittivity at each point of the dielectric.

We rearrange the above six equations in two groups of equations
for $(E_1,H_2,H_3)$ (transverse magnetic polarisation),
and $(H_1,E_2,E_3)$ (transverse electric polarisation).
We choose $E_1$ and $H_1$ as the unknown functions within the respective 
groups and notice that the remaining unknowns are expressed in terms of 
these two scalar functions only. The equations satisfied by $E_1,$ $H_1$ 
are 
\begin{equation}
(E_1')'+\bigl(\omega^2\mu\varepsilon
-\kappa^2\bigr)E_1=0,
\label{pcf2a}
\end{equation}
\begin{equation}
(\varepsilon^{-1}H_1')'+\bigl(\omega^2\mu
-\varepsilon^{-1}\kappa^2\bigr)H_1=0.
\label{pcf2b}
\end{equation}
Note that (\ref{pcf2b}) coincides with (\ref{ourproblem}) when $\kappa=0,$
by setting 
\[
\omega^2=\lambda,\ \ \mu=1,\ \ \varepsilon^{-1}(x_2)=p(x_2/\eta)\bigl(\eta^2\chi_0(x_2/\eta)
+\chi_1(x_2/\eta)\bigr),\ \ \ x_2\in(a,b),\ \ \eta>0,
\]
where we use $\eta$ rather than $\varepsilon$ to denote the structure period,
in order to avoid confusion with the standard notation for electric permittivity. Our analysis in Sections 
\ref{wholespacesection} and \ref{boundedintervalsection} carries 
over to the case $\kappa>0,$ where we get a $\kappa$-dependent version of the dispersion relation
(\ref{limiteq}), as follows:
\[
\frac{1}{2}(\alpha-\beta+1)\biggl(\sqrt{\lambda}-\frac{\kappa^2}{\sqrt{\lambda}}\biggr)
\sin\Bigl(\sqrt{\lambda}(\alpha-\beta)\Bigr)
+\cos\Bigl(\sqrt{\lambda}(\alpha-\beta)\Bigr)=\cos(2\pi\theta).
\]
Assuming infinitely conducting walls 
on either side of the dielectric (see {\it e.g.} \cite{pcfbook} for further details), we   
supply (\ref{pcf2a}) and (\ref{pcf2b}) with homogeneous Dirichlet and Neumann boundary
conditions respectively.

In the numerical solution of above problem we employ finite elements with perfectly 
matched layers, {\it i.e.} anisotropic
absorptive reflectionless layers (see {\it e.g.} \cite{pcfbook}), on the top and bottom of the computational domain.
Our results are shown in Figure \ref{siamfig1} for $\kappa=.1,$ $N=16,$ and for 
the transverse electric mode with frequency $\lambda_\star=78.34$ inside the third stop
band. The latter corresponds to the first trapped mode shown
in Figure \ref{crasfig2}(a), in view of the fact that for $\kappa=.1$ there is an additional 
zero-frequency stop band. 
The magnetic component of this mode (Figures \ref{siamfig1}(a) and (d)) clearly shares the same features
as the string mode in Figure \ref{crasfig2}(a).


\begin{center}
\begin{figure}[h]
\centering
\includegraphics[scale=0.3]{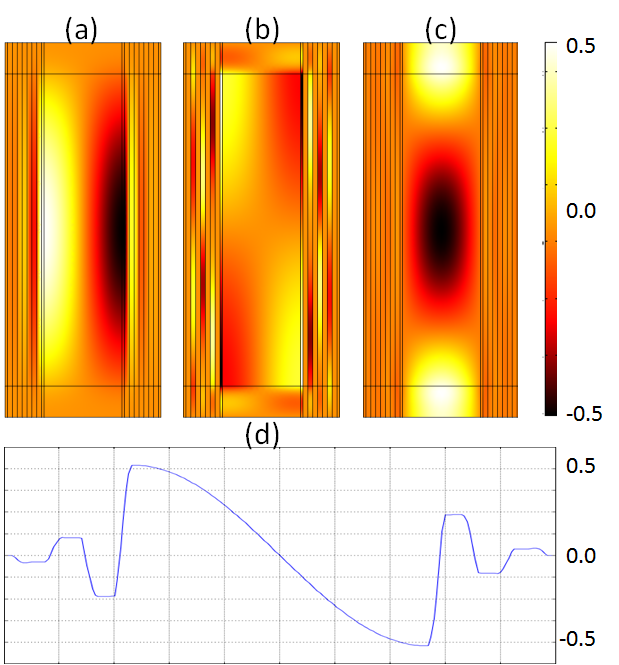}
\caption{Obliquely propagating transverse electric wave in a high-contrast dielectric multilayered planar
waveguide with infinitely conducting walls: (a) 2D plot of $H_1$; (b) 2D plot of $E_2$;
(c) 2D plot of $E_3$; (d) Profile of $H_1$ along the horizontal centreline.
Here $N=16$, $\kappa=.1$ and the normalized frequency $\lambda_\star=78.34$.}
\label{siamfig1} 
\end{figure}
\end{center}


\section*{Appendix A}

In this appendix we argue that for any fixed $N\in{\mathbb N}$ and $\lambda<0,$ say $\lambda=-1,$ the solutions to the problems (\ref{ourproblem}) converge in an appropriate two-scale sense (see {\it e.g.} \cite{Allaire}) to the solution of some limit problem parametrised by $N.$ 
Throughout our argument, we assume that $(a,b)={\mathbb R},$ although the results are valid for an arbitrary interval $(a,b)\subset{\mathbb R},$ irrespective of the boundary  conditions.

We first formulate the related statement for $N=1.$

\subsection{Periodic homogenisation}
\label{subph}

\begin{lemma}
\label{lem:hom1}
Set $\lambda=-1$ and let $u_\ep \in H^1({\mathbb R})$ be the solution to the problem (\ref{ourproblem}).
Then there exists $u(x,y)\in L^2(\Omega;V_1)$ such that
\[
u_\ep   \twoscale u(x,y), \ \ \ \ \ \ \ep u_\ep'\twoscale u_{,y}(x,y),\ \ \ \ \ \ \ \ 
\chi_1\left(\tfrac{x}{\ep}\right) p\left(\tfrac{x}{\ep}\right) u_\ep'(x)  \twoscale 0.
\]
Here the space $V_1$ is defined by 
\beq
V_1:=\left\{v\in H^1_{\#}(Q):p(y)v'(y)=0 {\ {\rm for}\ } y\in Q_1\right\}.
\eeq{V1def}

The function $u\in L^2({\mathbb R};V_1)$ is the unique solution to the problem
\[
\dblint{\mathbb R}{p(y)u_{,y}(x,y)\overline{\varphi'(y)\psi(x)}}
+\dblint{\mathbb R}{u(x,y)\overline{\varphi(y)\psi(x)}}
\]
\beq
= \dblint{\Omega}{f(x) \overline{\varphi(y)\psi(x)}},\quad \forall \varphi\in C^\infty_0(\mathbb R), \psi \in V_1. 
\eeq{eq:qh5}
\end{lemma}

Note that the $x$ dependence in the formulation \eqref{eq:qh5} is trivial, indeed $u(x,y)=f(x)v(y)$ where $v(y)\in V_1$ is the unique solution to the problem 
\begin{equation}
\label{eq:qh6}
\int_{Q_0}p(y)v'(y)\overline{\varphi'(y)}\mathrm{d}y+\int_Q v(y) \overline{\varphi(y)}\mathrm{d}y =\int_Q\overline{\varphi(y)}\mathrm{d}y, \quad \forall 
\varphi\in V_1. 
\end{equation}
This observation implies, in particular, that the spectra corresponding to 
\eqref{eq:qh5} and \eqref{eq:qh6} coincide. Therefore, denoting by $A_1$ the operator defined by the form 
\[
\mathfrak{b}_1(u,v)=\int_{Q_0}p(y)u'(y)\overline{v'(y)}\mathrm{d}y,\ \ \ u,v\in V_1,
\]
we obtain
$\lim_{\ep\rightarrow 0}S^\varepsilon\supset \sigma\left(A_1\right).$

\subsection{``$NQ$-periodic'' homogenisation}

In the argument above we found the two-scale limit operator $A_1$ by choosing the periodic reference cell to be 
$Q$ and passing to the two-scale limit in (\ref{ourproblem}) as $\ep \rightarrow 0.$  Replacing $Q$ with $NQ,$ 
$N\in{\mathbb N},$ we obtain an analogue of Lemma \ref{lem:hom1}, as follows.

\begin{lemma}
\label{lem:hom3}
Set $\lambda=-1$ and let $u_\ep$ be the solution to (\ref{ourproblem}). Then $u_\ep \twoscale u_N$, up to some sequence we do not relabel, where 
$u_N=f(x)v_N(y)$ and $v_N$ is the unique solution to
\[
\int_{NQ}\chi_0(y)p(y)(v_N)'(y)\overline{\varphi'(y)}\mathrm{d}y+\int_{NQ}v_N(y) 
\overline{\varphi(y)}\mathrm{d}y =\int_{NQ}\overline{\varphi(y)}\mathrm{d}y, \quad \forall \varphi \in V_N. 
\]
Here we denote 
\beq
V_N:=\left\{v\in H^1_{\#}(NQ):p(y)v'(y) = 0{\ {\rm for}\ }y \in Q_1\right\}.
\eeq{VNdef}
Furthermore, 
$\lim_{\ep \rightarrow 0} S^\varepsilon\supset \sigma\left(A_N\right),$
where $A_N$ is the operator defined using the bilinear form
\begin{equation}
\mathfrak{b}_N(u,v):=\int_{NQ}\chi_0(y) p(y)u'(y)\overline{v'(y)}\mathrm{d}y,\ \ \ \ \ u,v\in H^1_{\#}(NQ).
\end{equation}
\end{lemma}
Applying Lemma \ref{lem:hom3} for all $N\in \mathbb{N}$ yields 
\[
\lim_{\ep \rightarrow 0}\spec{A_\ep}\supset\bigcup_{N \in\mathbb{N}}\spec{A_N}.
\]

\subsection{Relation to the Bloch spectrum}

Notice that if $\theta=j/ N$ for some integers $N,j$ such that $0 \le j \le N-1$, then all eigenfunctions $v^k(\theta),$ $k=1,2,...,$ are $N$-periodic, in particular $v^k(\theta)\in V_N$. In fact,  $v^k(\theta),$ $k=1,2,...$ are 
eigenfunctions of $A_N.$ 
Indeed, for any fixed $\phi \in V_N$  define the function 
$\Phi(y):=\sum_{k=0}^{N-1} \phi(y+k)\exp(-2\pi{\rm i}\theta k)$
and notice that since $\Phi(y) \in V(\theta),$ one has $(A(\theta)v^k(\theta),\Phi) = \lambda_k(\theta)(v^k(\theta),\Phi)$. Therefore, writing for brevity $v^k(\theta)=v,$ we obtain
\[
\int_{NQ} \chi_0(y)p(y)v'(y) \overline{\varphi(y)} \mathrm{d}y=\sum_{k=0}^{N-1} \int_{Q}\chi_0(y+k)p(y+k)v'(y+k) \overline{\varphi(y+k)}\mathrm{d}y
\]
\[
=\int_{Q}\chi_0(y)p(y)v(y)\overline{\Phi(y)}\mathrm{d}y=
\lambda_k(\theta)\int_{Q}v(y)\overline{\Phi}(y)\mathrm{d}y
\]
\[
=\sum_{k=0}^{N-1}\lambda_k(\theta)\int_{Q}v(y)
\exp(2\pi{\rm i}\theta k)\overline{\phi(y+k)}\mathrm{d}y=\lambda_k(\theta)\int_{NQ}v(y)\overline{\phi(y)}\mathrm{d}y.
\]
The above observations show that
\[
\lim_{\ep \rightarrow 0}S^\varepsilon\supset\bigcup_{N\in{\mathbb N}}\sigma\left(A_N\right)\supset
\bigcup_{\substack{N\in\mathbb{N} \\ 0\le j \le N-1}}\sigma\bigl(A(j/N)\bigr).
\]
Using the facts that the set of rational numbers $j/N$ is dense in $[0,1)$ and 
that the eigenvalues $\lambda=\lambda(\theta)$ are continuous with respect to $\theta$ (see Appendix C below) yields
\[
\lim_{\ep \rightarrow 0}S^\varepsilon\supset\bigcup_{\theta\in[0,1)} \sigma\left(A(\theta)\right).
\]

\section*{Appendix B}


\subsection{One classical inequality}

Here, for reader's convenience, we give the proof of a version of the classical Poincar\'{e} inequality 
(see {\it e.g.} \cite{DuvautLions}), which we use in the present paper. 

\begin{lemma}
\label{prop111}
There exists a positive constant $C_{\rm P},$ which depends on $\alpha$ and $\beta$ only, such that 
\begin{equation}
\label{prop1:eq1}
\int_{Q} \vert u (y)\vert^2\mathrm{d}y \le C_{\rm P}\left(\left\vert \int_{Q_1} u(y)\mathrm{d}y \right\vert^2 + \int_{Q} \vert u'(y) \vert^2\mathrm{d}y \right), \quad \forall u \in H^1(Q).
\end{equation}
\end{lemma}

\begin{proof}
For fixed $x \in Q$, $y \in Q_1$, we have $\vert x - y \vert \le 1$ and
$$
u(x)-u(y) = \int^x_y u '(t)\mathrm{d}t ,
$$
which implies
$$
\vert u(x)\vert^2 + \vert u(y)\vert^2-2\Re(u(x)\overline{u(y)}) = \left\vert \int^x_y u'(t)\mathrm{d}t
\right\vert^2\le \vert x-y \vert \int^x_y \vert u '(t)\vert ^2\mathrm{d}t \le \int_{Q} \vert u '(y)
\vert^2\mathrm{d}y .
$$
Integrating first with respect to $x$, then with respect to $y,$ yields
\[
\vert Q_1 \vert \int_{Q} \vert u(x) \vert^2\mathrm{d}x + \int_{Q_1} \vert u(y) \vert^2\mathrm{d}y\ \ \ \ \ \ \ \ \ \ \ \ \ \ \ \ \ \ \ \ \ \ \ \ \ \ \ \ \ \ \ \ \ \ \ \ \ \ \ \ \  
\]
\[
\ \ \ \ \ \ \ \ \ \ \ \ \ \ \ \ \ \ \ \ \ \ \ \ \ \ \ \ \ \ \ \le 2\Re\left[\left( \int_{Q} u(x)\mathrm{d}x \right)  \overline{\left( \int_{Q_1} u(y)\mathrm{d}y \right)} \right] + \vert Q_1 \vert \int_{Q} \vert u '(y) \vert^2\mathrm{d}y . \label{prop1:eq2}
\]
Here we have used the fact that $\vert Q \vert = 1$. Now using the inequalities \eqref{prop1:eq2},  $ab\le \ep a^2 + b^2/\ep $ for any real $a,b$, $\ep >0$, and
$$
\left\vert \int_{Q_1} u(y)\mathrm{d}y\right\vert^2 \le \vert Q_1 \vert \int_{Q_1} \vert u(y)\vert^2\mathrm{d}y,
$$
we obtain
\begin{multline*}
\vert Q_1 \vert \int_{Q} \vert u(x) \vert^2\mathrm{d}x + \frac{1}{\vert Q_1 \vert}\left\vert \int_{Q_1} u(y)
\mathrm{d}y\right\vert^2  \le \vert Q_1 \vert \int_{Q} \vert u(x) \vert^2\mathrm{d}x + \int_{Q_1} \vert u(y) \vert^2\mathrm{d}y \\
 \le \ep\left( \int_{Q}  \vert u(y) \vert^2\mathrm{d}y \right) + \frac{1}{\ep} \left\vert \int_{Q_1} u(y)\mathrm{d}y \right\vert^2 + \vert Q_1 \vert\int_{Q} \vert u '(y) \vert^2\mathrm{d}y.
\end{multline*}
Setting $\ep = \vert Q_1 \vert/2$ gives
$$
\frac{\vert Q_1 \vert}{2}\int_{Q} \vert u(x) \vert^2\mathrm{d}x \le \frac{1}{\vert Q_1 \vert}\left\vert \int_{Q_1} u(y)\mathrm{d}y \right\vert^2 + \vert Q_1 \vert\int_{Q} \vert u '(y) \vert^2\mathrm{d}y.
$$
This is \eqref{prop1:eq1} for $C_{\rm P}=2 / \vert Q_1 \vert^2$, where $\vert Q_1 \vert=1-\beta+\alpha$.
\end{proof}

\subsection{The equivalence of Proposition \ref{prop0} and Proposition \ref{lem1}}

Suppose that Proposition \ref{lem1} holds and consider a bounded interval $D\subset{\mathbb R}$  and 
a function $u\in V(D)^\perp.$  For the proof of Theorem \ref{maintheorembounded} it is sufficient to consider 
the case when $D$ is an interval of integer length.  However, for completeness we carry out the argument for 
an arbitrary $D.$ 

Clearly, there is a positive integer $N$ and an interval $I:=(l,l+N)\supset\overline{D},$ 
$l\in{\mathbb R}$ such that $\{l,l+1\}\subset D_1.$ We extend the function $u$ to a function that is periodic on the interval $I$ and is such that $u\in V(I)^\perp,$ keeping the same notation $u$ for such an extension. 

Notice that for each $j=0,1,...,N-1$ the function 
\[
u^j(x)=N^{-1}\sum_{k=1}^Nu(x-l+k)\exp(-2\pi{\rm i}jk/N)
\]
belongs to the space $V(\theta_j)^\perp,$ $\theta_j:=j/N.$ Indeed, for any $v\in V(\theta_j)\subset H_{\theta_j}^1(Q)$ one has 
\beq
v(x)=v_{\#}(x)\exp({\rm i}\theta_j x),
\eeq{vper}
where $v_{\#}\in H_{\#}(Q).$ We extend $v_{\#}$ by periodicity to the 
whole of ${\mathbb R}$ and also extend $v$ to the whole of ${\mathbb R}$ so that the formula 
(\ref{vper}) holds for any $x\in{\mathbb R}.$ Then, for the extended function $v,$ one clearly has 
$v(\cdot+l)\in V(I),$ and $v(x-k)=v(x)\exp(-2\pi{\rm i}jk/N)$ for any $x\in{\mathbb R}.$ Therefore, for any 
$j=0,1,..., N-1,$ one has  
\[
\int_Qu^j(x)\overline{v(x)}dx=N^{-1}\sum_{k=0}^{N-1}\int_0^1u(x-l+k)\exp(-2\pi{\rm i}jk/N)\overline{v(x)}dx
\]
\[
=N^{-1}\sum_{k=0}^{N-1}\int_k^{k+1}u(x)\exp(-2\pi{\rm i}jk/N)\overline{v(x-k)}dx
\]
\[
=N^{-1}\int_0^Nu(x-l)\overline{v(x)}dx=N^{-1}\int_I u(x)\overline{v(x+l)}dx=0.
\]

Now, using the Parseval identity and Lemma \ref{prop111}, we obtain  
\[
\Vert u\Vert_{H^1(D)}^2\le\Vert u\Vert_{H^1(I)}^2=N\sum_{j=1}^N\Vert u^j\Vert_{H^1(Q)}^2\le N(C_{\rm P}+1)\sum_{j=1}^N\vert\vert\vert u^j\vert\vert\vert^2
\]
\[
\le(C_{\rm P}+1)\tilde{C}^2N\sum_{j=1}^N\Vert(u^j)'\Vert_{L^2(Q_1)}^2=(C_{\rm P}+1)\tilde{C}^2\Vert u'\Vert_{L^2(I_1)}^2,
\]
where $I_1:=I\cap F_1.$ By ensuring, when performing the above extension from $D$ to $I,$ that the inequality 
$\Vert u'\Vert_{L^2(I_1)}^2\le C\Vert u'\Vert_{L^2(D_1)}^2$ holds with a positive constant $C$ that does not depend on the interval $D,$
we complete the proof of the implication.
 
Conversely, assume the validity of Proposition \ref{prop0} and suppose that $\theta\in[0,1)$ is rational, $N\theta\in{\mathbb N}.$ For any $w\in V(\theta)^\perp,$ we notice 
that 
\beq
w(x)=u(x)\exp({\rm i}\theta x)
\eeq{wuform}
for some $u\in H^1_\theta(Q),$ and extend first the function $u$ by periodicity to the interval $D=NQ$ and then the function 
$w$ according to the formula (\ref{wuform}) to the same interval $D.$ Then using an argument similar to the above it is shown that 
$w\in V(D)^\perp,$ and hence
\[
\Vert w\Vert_{H^1(D)}\le C\Vert w'\Vert_{L^2(D_1)}
\]
for some positive constant $C.$ This automatically implies (\ref{lem1:eq1}) since the norm of $w$ is the same for any interval of length one. By continuity in 
$\theta$ the statement is extended to establish the existence of $C>0$ that serves all $\theta\in[0,1).$

\section{Appendix C}

The continuity of the family $V(\theta)$ implies that for a given $k$ dimensional subspace of $V(\theta)$ we can find, for $\theta'$ close to $\theta$, a $k$ dimensional subspace of $V(\theta')$ close to $V(\theta)$, in the sense of 
\cite{Kato}. More precisely,
\begin{proposition}
\label{cor:contV1}
Let $\theta_1\in[0,1)$ and $F_1 \subset V(\theta)$, $\text{\rm dim}F_1=k.$ For all $\ep>0$ there exists $\delta>0$ such that for all 
$\theta_2,$ $\vert\theta_2-\theta_1\vert\le\delta$, there exists $F_2 \subset V(\theta_2)$, $\text{\rm dim}F_2=k$ such that
$\max\{{\rm dist}(F_1,F_2),{\rm dist}(F_2,F_1)\}<\ep,$
where the distance between two linear subspaces $N,$ $M,$ of $H^1(Q)$, is defined by the formula
$$
{\rm dist}(N,M):=\sup_{\substack{u\in N\\ \norm{u}{H^1}=1}}\inf_{v\in M}\Vert u-v\Vert.
$$
\end{proposition}
Henceforth within this appendix, we write $H^1$ instead of $H^1(Q)$ for brevity.

\begin{proof}
Let $F_1 \subset V(\theta_1), \text{dim}F_1 = k,$ and let $f^1_n$, $n=1,\ldots,k,$ be a basis of $F_1$. Lemma \ref{contV} implies that for any  $\theta_2$ that is sufficiently close to $\theta_1,$ there exist $v_n \in V(\theta_2)$ that are close to
 $f^1_n$ in the $H^1$-norm. We construct an orthonormal sequence as follows:
\[
f^2_1 = v_1,\ \ f^2_2=a_{21} f^2_1+v_2,\ \ f^2_n = \sum_{m=1}^{n-1} a_{nm} f^2_m + v_n,\ n=3,\ldots,k,
\]
where $a_{nm}=-\norm{f^2_m}{H^1}^{-1}(v_n , f^2_m)_{H^1}.$
By construction, $f^2_1$ is close to $f^1_1$. Notice further that $a_{21} =\norm{v_1}{H^1}^{-1}(v_2,v_1)_{H^1}$ 
can be made as small as necessary if $\theta_2$ is sufficiently close to $\theta_1$. This implies that 
$f^2_2$ is close to $f^1_2$. By induction we show that $f^2_n$ is close to $f^1_n$ for all $n=1,2,...,k.$

Defining $F_2 = \text{span}\left\{ f^2_n : n=1,\ldots, k \right\}$, we argue that $F_2$ satisfies the properties of the proposition. Indeed, by construction,  $F_2 \subset V(\theta_2)$, dim$F_2=k$. Further, for fixed $u\in F_1$ one has $u = \sum_{n=1}^k b_n f^1_n$, for some $b_n \in \mathbb{C}$. The function $v= \sum_{n=1}^k b_n f^2_n$ belongs to $F_2$ and is close to $u$ if $\theta_2$ is sufficiently close to $\theta_1$. It follows that $\text{dist}(F_1, F_2)<\varepsilon$  when  $\vert\theta_2-\theta_1\vert<\delta$ for sufficiently small $\delta.$ 
Reversing the roles of $u$ and $v$ completes the proof. \end{proof}

\begin{theorem}
The eigenvalues $\lambda_k(\theta),$ $k\in{\mathbb N}$  are continuous functions of  $\theta\in[0,1).$
\end{theorem}

\begin{proof}
By a variational argument it is known that (see {\it e.g.} \cite{ReedSimon})
\begin{equation}
\label{eq:cV1}
\lambda_k(\theta) = \inf_{\substack{F \subset V(\theta) \\ \text{\rm dim}F = k }} \sup_{u \in F }\frac{\bigl(A(\theta)u,u\bigr)}{\Vert u\Vert_{L^2(Q)}^2},
\end{equation}
Let  $F_1 \subset V(\theta_1)$, $\text{\rm dim}F_1=k$. For a fixed $\ep >0$, let $\delta$, $\theta_2$ and $F_2 \subset V(\theta_2)$ be given by Proposition \ref{cor:contV1}. For a fixed $f_2 \in F_2$, there exists $f_1 \in F_1$ such that $\norm{f_1 - f_2}{H^1} \le \ep$. This implies
$$
\frac{\bigl(A(\theta_2)f_2,f_2\bigr)}{\norm{f_2}{L^2(Q)}^2}\le \frac{\bigl(A(\theta_1)f_1,f_1\bigr)}{\norm{f_1}{L^2(Q)}^2} + \ep \le\sup_{u\in F_1}\frac{\bigl(A(\theta_1)u,u\bigr)}{\norm{u}{L^2(Q)}^2}  + \ep.
$$
It follows, by the arbitrary choice of $f_2$, that
$$
\sup_{u\in F_2}\frac{\bigl(A(\theta_2)u,u\bigr)}{\norm{u}{L^2(Q)}^2}\le\sup_{u\in F_1}
\frac{\bigl(A(\theta_1)u,u\bigr)}{\Vert u\Vert_{L^2(Q)}^2}+\ep.
$$
Therefore, by the fact that $F_1$ is arbitrary and in view of  the equation \eqref{eq:cV1}, we obtain
\begin{align}
\inf_{\substack{F_2}} \sup_{u \in F_2 }\frac{\bigl(A(\theta_2)u,u\bigr)}{\norm{u}{L^2(Q)}^2}
 & \le\inf_{\substack{F_1 \subset V(\theta_1) \\ \text{dim}F_1=k}}\sup_{u\in F_1}\frac{\bigl(A(\theta_1)u,u\bigr)}{\norm{u}{L^2(Q)}^2} +\ep \nonumber \\
 & = \lambda_k(\theta)+\ep.\label{eq:cV2}
\end{align}
The set $\{F_2\}$ is a subset of the set of all $k$ dimensional subspaces of $V(\theta_2),$ therefore
\begin{equation}
\label{eq:cV4}
\lambda_k(\theta_2) =\inf_{\substack{F \subset V(\theta_2) \\ \text{\rm dim}F = k }}
\sup_{u\in F}\frac{\bigl(A(\theta)u,u\bigr)}{\norm{u}{L^2(Q)}^2}\le
\inf_{\substack{F_2}}\sup_{u\in F_2}\frac{\bigl(A(\theta_2)u,u\bigr)}{\norm{u}{L^2(Q)}^2}.
\end{equation}
Equations \eqref{eq:cV2} and \eqref{eq:cV4} imply
$$
\lambda_k(\theta_2)-\lambda_k(\theta_1)\le\ep.
$$
Reversing the roles of $\theta_1$ and $\theta_2$ yields the desired result.
\end{proof}

The above statement of continuity of $\lambda_k=\lambda_k(\theta),$ $k\in{\mathbb N},$ is not a 
simple consequence of the continuity of eigenvalues for the usual Floquet-Bloch decomposition. 
An important distinct feature of the present statement is the dependence on $\theta$ of the operator 
domain $V(\theta)$ in the variational principle (\ref{eq:cV1}).

\end{document}